\documentclass[11pt,a4paper]{article}

\usepackage[english]{babel}
\usepackage{times}
\usepackage{graphicx}
\usepackage{amscd}
\usepackage{amsmath}
\usepackage{amsrefs}
\usepackage{amsfonts}
\usepackage{amssymb}
\usepackage{amsthm}
\usepackage{latexsym}

\newtheorem{theorem}{Theorem}[section]
\newtheorem{proposition}[theorem]{Proposition}
\newtheorem{lemma}[theorem]{Lemma}
\newtheorem{corollary}[theorem]{Corollary}
\newtheorem{definition}[theorem]{Definition}
\newtheorem{remark}[theorem]{Remark}

\numberwithin{equation}{section}

\newcommand{\Z}{\mathbb{Z}}
\newcommand{\R}{\mathbb{R}}
\newcommand{\N}{\mathbb{N}}
\renewcommand{\epsilon}{\varepsilon}
\newcommand{\eps}{\varepsilon}
\newcommand{\e}{\varepsilon}

\newcommand{\ugu}{\;{\stackrel{k}{=}}\;}

\renewcommand{\leq}{\leqslant}
\renewcommand{\le}{\leqslant}

\renewcommand{\ge}{\geqslant}

\usepackage{authblk}

\begin{document}

\author[1]{Serena Dipierro}

\author[2]{Ovidiu Savin}

\author[1,3,4,5]{Enrico Valdinoci}

\affil[1]{\footnotesize School of Mathematics and Statistics,
University of Melbourne,
Richard Berry Building,
Parkville VIC 3010,
Australia \medskip } 

\affil[2]{\footnotesize Department of Mathematics, Columbia University,
2990 Broadway,
New York NY 10027, USA \medskip}

\affil[3]{\footnotesize Weierstra{\ss} Institut f\"ur Angewandte
Analysis und Stochastik, Hausvogteiplatz 5/7, 10117 Berlin, Germany \medskip}

\affil[4]{\footnotesize Dipartimento di Matematica, Universit\`a degli studi di Milano,
Via Saldini 50, 20133 Milan, Italy \medskip }

\affil[5]{\footnotesize Istituto di Matematica Applicata e Tecnologie Informatiche,
Consiglio Nazionale delle Ricerche,
Via Ferrata 1, 27100 Pavia, Italy \medskip }

\title{Definition of fractional Laplacian\\
for functions with polynomial growth\thanks{The second author has been supported by NSF grant DMS- 1200701. The third author has been supported by ERC grant 277749 ``EPSILON Elliptic Pde's and Symmetry of Interfaces and Layers for Odd Nonlinearities'' and PRIN grant 201274FYK7 ``Aspetti variazionali e perturbativi nei problemi differenziali nonlineari''.
It is a pleasure to thank Susanna Terracini for her interesting comments.
Part of this paper has been written on the occasion of a very pleasant visit
of the first and third authors to Columbia University.
An informal talk about this paper has been given
on the occasion
of the Conference ``Asymptotic Patterns in Variational Problems: PDE and Geometric Aspects''
in Oaxaca, and a video is available at
{\tt http://www.birs.ca/events/2016/5-day-workshops/16w5065/videos/watch/201609281545-Valdinoci.html}
Emails:
{\tt serydipierro@yahoo.it}, {\tt savin@math.columbia.edu}, {\tt enrico@mat.uniroma3.it} }}

\date{}

\maketitle

\begin{abstract}
We introduce a notion of fractional Laplacian
for functions which grow more than linearly at infinity. In such case,
the operator is not defined in the classical sense: nevertheless,
we can give an ad-hoc definition which can be
useful for applications in various fields,
such as blowup and free boundary problems.

In this setting, when the solution has a polynomial
growth at infinity, the right hand side of the equation
is not just a function, but an equivalence class of functions
modulo polynomials of a fixed order.

We also give a sharp version of the Schauder estimates in this framework,
in which the full smooth H\"older norm of the solution is controlled
in terms of the seminorm of the nonlinearity.

Though the method presented is very general
and potentially works for general nonlocal operators, for clarity and concreteness
we focus here on the case of the fractional Laplacian.
\end{abstract}

\bigskip\bigskip

\section{Introduction}

As well known (see e.g.~\cite{MR0214795, MR2707618, MR2944369}),
to define the fractional Laplacian of a function\footnote{For short,
in the rest of the paper,
the principal value notation in \eqref{DE:s} will be
tacitly understood and not repeated.} as
\begin{equation} \label{DE:s}
(-\Delta)^s u(x) := \lim_{\e\to0}\int_{\R^n\setminus B_\e(x)}
\frac{u(x)-u(y)}{|x-y|^{n+2s}}\,dy,\end{equation}
with~$s\in(0,1)$, two types of assumptions are needed, namely:
\begin{itemize}
\item the function $u$ needs to be sufficiently regular near~$x$,
\item the function $u$ needs to have a growth control at infinity.
\end{itemize}
The regularity condition is indeed needed in order to make
the integral in~\eqref{DE:s} convergent near the singularity
(possibly after cancellation). On the other hand,
the growth condition at infinity is needed to make the
tail of the integrand convergent: for this scope, usually
the most general assumption on~$u$ at infinity can be written in
the form
\begin{equation}\label{G:sta}
\int_{\R^n} \frac{|u(y)|}{1+|y|^{n+2s}}\,dy < +\infty
.\end{equation}
The need of assumptions at infinity is a typical feature of fractional
problems (of course, in the case of the classical Laplacian, 
there is no need to prescribe this kind of conditions in order to compute
derivatives). In this sense, the study of nonlocal operators
presents several conceptual difficulties with respect to the
classical case, inherited from
the fact that the behavior at infinity
may deeply affect the value of the fractional Laplacian:
see e.g.~\cite{2014arXiv1404.3652D, 2016arXiv160904438D}
for rather general examples (in particular, roughly speaking,
appropriate oscillations
at infinity can make the fractional Laplacian vanish identically
in a given ball, basically independently on the values of the function
in such a ball).\medskip

In addition, conditions at infinity such as~\eqref{G:sta}
often provide a series of additional difficulties in the regularity
theories for fractional operators since this type of assumptions
behaves badly with respect to scalings and blowups:
as an example, one can consider a function which is bounded
and quadratic near the origin and check that its blowup
does not satisfy~\eqref{G:sta} (in spite of the fact that both
the original function and its blowup may be as smooth as we wish).\medskip

The goal of this paper is to provide a natural setting
to make sense of the fractional Laplacian under weaker conditions
at infinity. Of course, some condition at infinity must be
taken in order to avoid the examples in~\cite{2014arXiv1404.3652D, 2016arXiv160904438D},
nevertheless we give here a framework which is more flexible and
compatible with scalings and blowups.\medskip

The basic idea for this is that, if the function grows too much at infinity,
its fractional Laplacian diverges, but it can be written as a given function
``plus a diverging sequence of polynomials\footnote{As customary,
polynomials of negative degree are set to be zero.}
of a given degree''.
For instance, if the function grows linearly at infinity and~$s=1/2$,
then condition~\eqref{G:sta} is violated and~$\sqrt{-\Delta}$ cannot
be defined in the usual sense. We will see that, in this case,
a definition is possible,
up to ``a diverging sequence of constants''.\medskip

{F}rom this, one is formally allowed to ``take derivatives of the equation''
and obtain regularity estimates: in the previous example,
one would say that the derivatives of the constants play no role
and,
in case one has ``polynomials of degree~$k-1$ as a remainder'',
the equation will be well posed ``up to derivatives of order~$k$''
(which make these polynomials vanish). Clearly, a rather delicate
argument will be used to check that this formal idea makes sense at
all, since these additional ``remainders'' are divergent and so they
do not obviously vanish after differentiation.\medskip

As a matter of fact, to introduce the general setting of possibly divergent
fractional Laplacians and to develop the related regularity theory,
we will use sequences of cutoffs to reduce the problem to the more usual
setting and we will obtain uniform estimates in an appropriate sense.
To this aim, we consider\footnote{The choice of the particular
cutoff in~\eqref{9d01} has been made for the sake of concreteness.
Other choices are indeed possible as well.}
the family of cutoffs
\begin{equation}\label{9d01}
\chi_R(x):=\left\{ \begin{matrix}
1 & {\mbox{ if }} x\in B_R,\\
0 & {\mbox{otherwise,}}
\end{matrix}
\right.
\end{equation}
and we
fix the following setting.

\begin{definition}\label{DE:k}
Let $s\in(0,1)$, $k\in\N$, $u:\R^n\to\R$ and $f:B_1\to\R$.
Assume that~$u$ is continuous in~$B_1$ and
\begin{equation}\label{G:k}
\int_{\R^n} \frac{|u(y)|}{1+|y|^{n+2s+k}}\,dy < +\infty
.\end{equation}
We say that
$$ (-\Delta)^s u \ugu f \qquad{\mbox{ in }}\;B_1$$
if there exist
a family of polynomials $P_R$, which have degree at most $k-1$,
and functions~$f_R:B_1\to\R$
such that
\begin{equation} \label{98e28jkgf12345s}
(-\Delta)^s (\chi_R u) = f_R + P_R\end{equation}
in~$B_1$ in the viscosity sense, with
\begin{equation}\label{9873irtugefjhdvdghaiugdgf234:00}
\lim_{R\to+\infty} f_R(x) =  f(x)\end{equation}
for any $x\in B_1$.
\end{definition}

We stress
again the fact that a classical definition of $(-\Delta)^s$
is not available in the setting of Definition~\ref{DE:k}
(not even for smooth functions)
unless one requires condition~\eqref{G:sta}
(and, of course, the condition in \eqref{G:k} is weaker 
than the one in \eqref{G:sta}
when $k\ge1$).
In this sense, the notation $(-\Delta)^s u$ in the case of Definition~\ref{DE:k}
represents a ``divergent'' operator. Nevertheless, as we will see in the forthcoming
Corollary \ref{POLLI2}, it is always possible to construct a function $f$
as requested by Definition~\ref{DE:k} (in particular,
the set of functions $u$ for which Definition~\ref{DE:k} makes sense is non-void).\medskip

Also, as it will be discussed in Corollary \ref{8IJAJS2wedfv1234},
the notion given by Definition~\ref{DE:k} reduces to the standard
fractional Laplacian when $k=0$.\medskip

Moreover, it follows easily from
Definition~\ref{DE:k} that
\begin{equation}\label{9idwfjdv827rutgjsakhdaaaaa2} {\mbox{if $
(-\Delta)^s u \ugu f$ in $B_1$,
then
$
(-\Delta)^s u\;
{\stackrel{k+1}{=}} \; f$ in $B_1$}}.\end{equation}
In terms of applications, we mention that
condition \eqref{G:sta} is often ``too rigid''
in the nonlocal framework: for instance, in many free boundary problems,
it is important to look at blowup sequences with degree higher than one
(say, $3/2$), and the blowup will not satisfy \eqref{G:sta}, see e.g. \cite{2015arXiv150405569D, 2016arXiv160105843C}.
Some ad-hoc arguments have been sometimes exploited in the literature
to overcome this type of difficulties, but we believe that a setting as the one
in Definition \ref{DE:k} can provide technical simplifications
and conceptual advantages when dealing with these cases.
\medskip

We also observe that the function $f$ in Definition \ref{DE:k}
is not uniquely determined, since any fixed polynomial can
be added to $f_R$ (and subtracted to $P_R$)
without affecting the setting in Definition \ref{DE:k},
and so
\begin{equation}\label{0dfhvnoooP}
{\mbox{if $(-\Delta)^s u\ugu f$ in $B_1$,
then $(-\Delta)^s u\ugu f+P$ in $B_1$
for any polynomial of degree $k-1$.}}  
\end{equation}
Nevertheless, the multiplicity in \eqref{0dfhvnoooP}
is exactly the one which characterizes $f$. Namely, we have that $f$
is determined up to polynomials of degree $k-1$, as pointed out by the following
observation (whose simple proof is given in Section~\ref{SECT:2}):

\begin{lemma}\label{UGUP}
Assume that $(-\Delta)^s u\ugu f$ and $(-\Delta)^s u\ugu \tilde f$ in $B_1$.
Then, there exists a polynomial $P$ of degree at most $k-1$ such that $f-\tilde f=P$.
\end{lemma}

As an illustrative example of our setting, let us point out that
one can compute $\sqrt{-\Delta}\, u$ when $u(x)=x^2$ in dimension $n=1$,
using Definition \ref{DE:k} with $k=2$. Indeed by a direct computation, one sees
that
\begin{equation}\label{801js78ajsjjs}
\sqrt{-\Delta}\, x^2 \;{\stackrel{2}{=}}\;0,
\end{equation}
and in fact a more general result will be presented in
Theorem \ref{LIOUVILLE} below.

Of course, from \eqref{801js78ajsjjs} and
Lemma \ref{UGUP}, a bunch of ``curious'' identities follows, such as
\begin{equation}\label{iosucjv43546758ikdjsbx13}\begin{split}
&\sqrt{-\Delta}\, x^2 \;{\stackrel{2}{=}}\;0, \qquad\qquad
\sqrt{-\Delta} \,x^2 \;{\stackrel{2}{=}}\;1, \qquad\qquad
\sqrt{-\Delta} \,x^2 \;{\stackrel{2}{=}}\;-1, \qquad\qquad
\\&\qquad\qquad\sqrt{-\Delta} \,x^2 \;{\stackrel{2}{=}}\; x \qquad\qquad
\sqrt{-\Delta} \,x^2 \;{\stackrel{2}{=}}\;ax+b, \qquad\qquad
\end{split}\end{equation}
for any $a$, $b\in\R$ (these identities indeed look funny at a first glance,
nevertheless they are all correct
in our setting).\medskip

A counterpart of our construction could be also
discussed in terms of extension results and Dirichlet-to-Neumann operators.
For instance, if one looks for the general harmonic function $U=U(x,y)$
in $\R\times(0,+\infty)$ with $U(x,0)=x^2$ and with at most quadratic growth
at infinity, one has that $U(x,y)=x^2-y^2-axy-by$.
In this sense, one is tempted (as usual) to identify
$\sqrt{-\Delta} x^2$ with $-\partial_y U(x,0)=ax+b$, for any $a$, $b\in\R$
which is indeed the last identity in \eqref{iosucjv43546758ikdjsbx13}.\medskip

As a matter of fact, 
an alternative approach 
to the one given in Definition \ref{DE:k}
would consist
in considering an extension problem
(modulo polynomials), 
but we followed the procedure
in Definition \ref{DE:k}, since it can be applied
to more general kernels.\medskip

Besides the intrinsic beauty of identities such as the ones
in \eqref{iosucjv43546758ikdjsbx13},
in our framework, the usefulness of Definition \ref{DE:k}
lies in its flexibility and possibility of applications
to obtain sharp regularity estimates. In this sense,
we give the following result, which can be seen as an optimal
bound in H\"older spaces for the derivatives of the solution
in terms of the seminorm of its (possibly divergent) fractional Laplacian
and a weak control of the function at infinity, as given in~\eqref{G:k}.

To this aim, 
as usual, if~$m\in\N$,
$\theta\in(0,1)$ and~$\gamma=m+\theta$,
we use the notation
\begin{eqnarray*}&&
\|f\|_{C^m(B_1)}:=\sum_{j=0}^m \| D^j f\|_{L^\infty(B_1)},\\
&& [f]_{C^\gamma(B_1)} :=\sup_{x\ne y\in B_1}
\frac{|D^mf(x)-D^mf(y)|}{|x-y|^\theta}\\
{\mbox{and }}&&
\|f\|_{C^\gamma(B_1)} := \|f\|_{C^m(B_1)} + [f]_{C^\gamma(B_1)} .
\end{eqnarray*} 
It is also convenient to introduce the following $k$-convention
on H\"older norms: we denote 
\begin{equation}\label{CGAMMA}
[ f ]_{C^\gamma(\Omega;k)}:=\inf [f-P]_{C^\gamma(\Omega)},
\end{equation}
where the $\inf$ is taken over all the polynomials $P$ of degree at most $k-1$;
of course, when $\gamma>k-1$, these polynomials
disappear after derivation and we have that 
\begin{equation}\label{TOGLI}
[ f ]_{C^\gamma(\Omega;k)}= [ f ]_{C^\gamma(\Omega)}
\qquad{\mbox{ if }}\gamma>k-1.\end{equation}
Notice that the setting in \eqref{CGAMMA} is consistent with
the multiplicity in \eqref{0dfhvnoooP}, since, for any
polynomial $Q$ of degree at most $k-1$, we have that
$$ [ f +Q]_{C^\gamma(\Omega;k)}=[ f ]_{C^\gamma(\Omega;k)}.$$
With this notation,
the precise statement of our Schauder 
estimates\footnote{Throughout this paper, we will
use the standard notation for the complementary set. Namely, given~$X\subseteq\R^n$
we set~$X^c:=\R^n\setminus X$.}
is the following.

\begin{theorem}\label{SCHAUDER}[$k$th order Schauder estimates
for divergent fractional Laplacians]
Let $s\in(0,1)$, $k\in\N$ and $u:\R^n\to\R$.

Assume that $u$ is continuous in~$B_1$ and
\begin{equation*}
J_{u,k}:=\int_{B_{1/2}^c} \frac{|u(y)|}{|y|^{n+2s+k}}\,dy < +\infty
.\end{equation*}

Suppose that
$$ (-\Delta)^s u \ugu f \qquad{\mbox{ in }}\;B_1.$$
Then, for any 
$\gamma>0$
such that~$\gamma\not\in\N$ and
$\gamma+2s\not\in\N$, and any $\ell\in \N$, it holds that
\begin{equation}\label{GH:OAA}
\|u\|_{C^{\gamma+2s}(B_{1/2})}
\le C\, \Big( [f]_{C^\gamma(B_1;\ell)}+J_{u,\ell}
\Big),
\end{equation}
for some $C>0$, only depending on $n$, $s$, $\gamma$,
$k$ and $\ell$.\end{theorem}

We remark that, differently from the usual way of writing the
Schauder estimates, the right hand side of \eqref{GH:OAA} does not
contain $\|u\|_{L^\infty(B_1)}$ nor $\|f\|_{C^\gamma(B_1)}$.
That is, we can bound
the whole norm $\|u\|_{C^{\gamma+2s}(B_{1/2})}$
with a contribution of $u$ coming from
outside $B_{1/2}$, which is encoded in the term $J_{u,\ell}$, and the oscillation
of $f$ in the seminorm~$[f]_{C^\gamma(B_1;\ell)}$.

In this sense, 
Theorem \ref{SCHAUDER} not only applies
to divergent operators, but it
is also a sharp version of the Schauder estimates
for non-divergent operators (notice indeed that when~$k=0$,
the setting of Theorem \ref{SCHAUDER} reduces to the one
of the classical fractional Laplace equation, and in this case
Theorem \ref{SCHAUDER} provides already a sharp result, compare e.g.
with Theorem 6 in \cite{MR3331523},
Theorem 1.1 in~\cite{MR3482695},
Proposition~7.1 in~\cite{2016arXiv160102548C} and the references therein).\medskip

A simple, but rather instructive consequence of
Theorem \ref{SCHAUDER} is a uniform bound on polynomial nonlinearities
in which the nonlinearity does not appear explicitly on the right hand side
(but it affects the size of $u$ near the boundary of the domain):

\begin{corollary}\label{78zj1001udhf}
Let $s\in(0,1)$, $k\in\N$ and $u:\R^n\to\R$.

Assume that $u$ is continuous in~$B_1$ and
\begin{equation*}
J_{u,k}:=\int_{B_{1/2}^c} \frac{|u(y)|}{|y|^{n+2s+k}}\,dy < +\infty
.\end{equation*}
Suppose that $f$ is a polynomial of degree $d$ and
$$ (-\Delta)^s u \ugu f \qquad{\mbox{ in }}\;B_1.$$
Then, for any
$\gamma>0$
such that~$\gamma\not\in\N$ and
$\gamma+2s\not\in\N$, it holds that
$$ \|u\|_{C^{\gamma+2s}(B_{1/2})}
\le C\, J_{u,d+1},$$
for some $C>0$, only depending on $n$, $s$, $\gamma$,
$k$ and $d$.
\end{corollary}

We observe that Corollary \ref{78zj1001udhf} is indeed an immediate consequence
of Theorem \ref{SCHAUDER}, by taking $\ell:=d+1$ there.
As a matter of fact, Corollary \ref{78zj1001udhf} is new, to the best of
our knowledge, even in the case $k=0$ corresponding to the standard fractional Laplacian.
\medskip

We also say that~$(-\Delta)^s u\ugu f$ in~$\R^n$
if the setting of Definition~\ref{DE:k} holds true in~$B_M$
(instead of~$B_1$), for all~$M>0$.
As a consequence of Theorem \ref{SCHAUDER}, we also
obtain a rigidity and classification result for
possibly divergent $s$-harmonic functions, as given here below.

\begin{theorem}\label{LIOUVILLE}[Liouville Theorem
for divergent fractional Laplacians]
Let $s\in(0,1)$, $k\in\N$ and $u:\R^n\to\R$.

Assume that $u$ is continuous and
\begin{equation*}
\int_{\R^n} \frac{|u(y)|}{1+|y|^{n+2s+k}}\,dy < +\infty
.\end{equation*}
Let
\begin{equation} \label{dksdef} d(k,s):=
\left\{\begin{matrix} k+1 & {\mbox{ if }} s\in\left( \frac12,\,1\right),\\
k & {\mbox{ if }} s\in\left( 0,\frac12\right] .\end{matrix}
\right.\end{equation}
Then, 
$$ (-\Delta)^s u \ugu 0 \qquad{\mbox{ in }}\;\R^n$$
if and only if $u$ is a polynomial of degree at most $d(k,s)$.
\end{theorem}

We recall that the study of rigidity properties for solutions of
nonlocal equations and related Liouville results are a very active field
of research, and this type of
results has also
important consequences on several aspects of the regularity theory,
see e.g.
\cite{MR2707618, MR3348929, MR3482695, MR3477075, MR3511811, 2015arXiv151206509F} and the references
therein. As far as we know, Theorem~\ref{LIOUVILLE} is the first result of
this type which takes into account the case of possibly divergent operators.
\medskip

We also point out that the notion given in Definition~\ref{DE:k}
is stable under limits, as given by the following result:

\begin{theorem}\label{STAB}[Stability of divergent fractional Laplacians]
Let $s\in(0,1)$ and $k\in\N$.
Let us consider sequences of functions $u_m:\R^n\to\R$ and $f_m:B_1\to\R$ such that
$u_m$ and $f_m$ are continuous in~$B_1$, 
and
\begin{equation}\label{78uidaaa3r9quwehfg}
(-\Delta)^s u_m \ugu f_m \qquad{\mbox{ in }}\;B_1.\end{equation}
Assume that $u_m\to u$ in~$L^1(B_1)$
and locally uniformly in $B_1$,
and that~$f_m\to f$ locally uniformly in $B_1$
as $m\to +\infty$, for some functions $u\in L^1(\R^n)$ and $f:B_1\to\R$.

Suppose also that
\begin{equation}\label{iodsjhdhsuususu4857489:PRE}
\sup_{m\in\N} \int_{\R^n} \frac{|u_m(y)|}{1+|y|^{n+2s+k}}\,dy<+\infty\end{equation}
and that~$u_m$ converges to~$u$ weakly in the following sense:
\begin{equation}\label{iodsjhdhsuususu4857489}
\lim_{m\to+\infty} \int_{\R^n} \frac{u_m(y)
\,\varphi(y)}{1+|y|^{n+2s+k}}\,dy=\int_{\R^n} \frac{u(y)
\,\varphi(y)}{1+|y|^{n+2s+k}}\,dy,\end{equation}
for any $\varphi\in L^\infty(\R^n)$.

Then, it holds that
\begin{equation}\label{9iuwde83ueghhfhfhfhfh11} 
(-\Delta)^s u\ugu f \qquad{\mbox{ in }}\;B_1.\end{equation}
\end{theorem}

Theorem~\ref{STAB} is the counterpart, in our setting,
of classical approximation and stability results in the fractional setting,
see~\cite{MR2781586}.\medskip

The rest of the paper is organized as follows.
In Section~\ref{SECT:2}, we recall 
some ancillary results
on polynomials and we prove Lemma~\ref{UGUP}.
In Section~\ref{SECT:1}, we compute the fractional Laplacian of
a cutoff function and we expand its possibly divergent behavior for
a family of cutoffs, showing that this procedure is compatible with
Definition~\ref{DE:k} and we provide a series of consistency
results between Definition~\ref{DE:k} and the standard
fractional Laplacian, when the two settings overlap.
 
Then, we provide
the proof of Theorem~\ref{SCHAUDER} in Section~\ref{SECT:3}.
This in turn will allow us to prove
Theorem~\ref{LIOUVILLE}
in Section~\ref{SECT:4}.
The proof of Theorem \ref{STAB} 
is given in Section \ref{STABS}.
The paper ends with some auxiliary appendices.

\section{Some remarks on polynomials}\label{SECT:2}

Here we recall the following
elementary, but useful, algebraic observations (the standard
proofs, for the convenience of the reader, are given in Appendix~B):

\begin{lemma}\label{9uids2345hv}
Let $P^{(j)}$ be a sequence of polynomials of degree at most~$d-1$.
Assume that there exist a bounded, open 
set~$U\subseteq\R^n$ and a function~$F:U\to\R$
such that
\begin{equation}\label{6y89ufidhgg177171} \lim_{j\to+\infty} P^{(j)}(x)=F(x)\end{equation}
for any~$x\in U$. Then, $F$ is a polynomial of degree at most~$d-1$
and the convergence in \eqref{6y89ufidhgg177171} holds in $C^m(U)$ for any $m\in\N$.
\end{lemma}

We also provide a variant of Lemma~\ref{9uids2345hv}, 
which will be used in the proof of Theorem~\ref{STAB}. 
For this, we introduce some notation: 
for any polynomial $P$, let~$U\subseteq\R^n$ be 
a bounded, open set with smooth boundary and define
\begin{equation}\label{no0195678tif}
\| P\|_\star:= \sup_{{\varphi\in C^2_0(U)}\atop{\|\varphi\|_{C^2(\R^n)}
\le 1}} \int_{U} P(x)\,\varphi(x)\,dx.\end{equation}
Then, we have the following convergence result:

\begin{lemma}\label{45678959uids2345hv}
Let $P^{(j)}$ be a sequence of polynomials of degree at most~$d-1$.
Assume that $P^{(j)}$ is a Cauchy sequence in the norm~$\| \cdot\|_\star$.
Then, there exists a polynomial~$P$ of degree at most~$d-1$
such that~$P^{(j)}$ converges to $P$ uniformly in~$U$
as~$j\to+\infty$.
\end{lemma}

With Lemma~\ref{9uids2345hv}, we can give the proof of Lemma \ref{UGUP}, by arguing as follows.

\begin{proof}[Proof of Lemma \ref{UGUP}] {F}rom
Definition \ref{DE:k}, we know that
there exist two families of polynomials $P_R$ and $\tilde P_R$, with degree at most $k-1$,
such that,
for any $x\in B_1$,
\begin{eqnarray*}
&& \lim_{R\to+\infty} (-\Delta)^s (\chi_R u)(x)- P_R(x) =  f(x)\\
{\mbox{and }}&&
\lim_{R\to+\infty} (-\Delta)^s (\chi_R u)(x)- \tilde P_R(x) =  \tilde f(x).\end{eqnarray*}
Accordingly, 
$$ f(x)-\tilde f(x)=\lim_{R\to+\infty} P_R(x)-\tilde P_R(x).$$
Since $P_R-\tilde P_R$ is a polynomial of degree at most $k-1$,
we deduce from Lemma~\ref{9uids2345hv} that $f-\tilde f$ is also a polynomial of degree at most $k-1$, as desired.
\end{proof}

We also give the following rigidity result (for general
unique continuation principles in the nonlocal setting, see also~\cite{MR3169789}).

\begin{lemma}\label{PLO:0}
Let~$R>r>0$.
Let~$P$ be a polynomial and $u$ be a viscosity solution of~$(-\Delta)^s u=P$
in~$B_R$. Assume that~$u=0$ in~$B_r^c$. Then~$u$ vanishes identically.
\end{lemma}

\begin{proof} We argue by induction on the degree~$d$ of~$P$.
If~$d=-1$, then~$P$ vanishes identically and the claim follows from the maximum
principle.

Suppose now the claim true for all polynomials of degree~$d-1$.
Let~$r'$, $R'\in (r,R)$ with~$R'>r'$.
For~$\theta\in\R^n$, with~$|\theta|$ sufficiently small,
we see that the function~$u^{(\theta)}(x):=u(x+\theta)-u(x)$
satisfies~$(-\Delta)^s u^{(\theta)}=P^{(\theta)}$ in~$B_{R'}$,
with~$P^{(\theta)}(x):=P(x+\theta)-P(x)$, and~$u^{(\theta)}=0$
outside~$B_{r'}$. We observe that~$P^{(\theta)}$
is a polynomial of degree at most~$d-1$, hence, by inductive hypothesis,
it follows that~$u^{(\theta)}$ is identically zero,
and therefore~$u$ is constant.

Since~$u$ vanishes outside~$B_r$, it thus follows that it vanishes everywhere,
as desired.
\end{proof}

\section{The role of the cutoff for divergent fractional Laplacians}\label{SECT:1}

In this section, we show how a cutoff affects
the computation of the fractional Laplacian
for a function with prescribed growth at infinity.
We will see that the identities obtained are compatible with
the setting in Definition \ref{DE:k}, namely the growth at infinity,
combined with a cutoff, produces a family of polynomials of a fixed degree.

\begin{theorem}\label{POLLI}
Let $s\in(0,1)$, $k\in\N$ and $u:\R^n\to\R$.

Assume that $u\in C^\alpha_{\rm loc}(B_1)$ for some $\alpha>2s$ and
\begin{equation}\label{INTEzz}
\int_{\R^n} \frac{|u(y)|}{1+|y|^{n+2s+k}}\,dy < +\infty
.\end{equation}
Let $\tau:\R^n\to\R$ be compactly supported and with $\tau=1$ in $B_2$.
Then, there exist a function $f_{u,\tau}:\R^n\to\R$,
and a polynomial $P_{u,\tau}$,
which has degree at most $k-1$,
such that
\begin{equation}\label{11u8u2139}
(-\Delta)^s (\tau u)= P_{u,\tau}+f_{u,\tau}\end{equation}
in $B_1$.

In addition, $f_{u,\tau}$ can be written in the following form:
there exists $\psi:B_1\times B_2^c\to\R$, with
\begin{equation}\label{PSI BUONA}
\sup_{{x\in B_1},\,{y\in B_2^c}} |\partial^\gamma_x \psi(x,y)| < +\infty\end{equation}
for any $\gamma\in\N^n$, such that
\begin{equation}\label{FU0}
f_{u,\tau}=f_{1,u}+f_{2,u}+f^\star_{u,\tau},\end{equation}
where
\begin{equation}\label{f3dif}
\begin{split}
& f_{1,u}(x):= \int_{B_2} \frac{u(x)-u(y)}{|x-y|^{n+2s}}\,dy,\\
& f_{2,u}(x):= \int_{B_2^c} \frac{u(x)}{|x-y|^{n+2s}}\,dy\\{\mbox{and }}\quad&
f^\star_{u,\tau}(x):=\int_{B_2^c} \frac{\tau(y)\, u(y)\;\psi(x,y)}{|y|^{n+2s+k}}\,dy.
\end{split}\end{equation}
\end{theorem}

\begin{proof} We stress that
the integral defining $f^\star_{u,\tau}$ is finite, thanks to \eqref{INTEzz}.

Now we compute, for any $x\in B_1$, 
\begin{equation} \label{POLx1}
\begin{split}
(-\Delta)^s (\tau u)(x) \,&=\int_{B_2} \frac{u(x)-u(y)}{|x-y|^{n+2s}}\,dy
+\int_{B_2^c} \frac{u(x)-(\tau u)(y)}{|x-y|^{n+2s}}\,dy\\
&=f_{1,u}(x)+f_{2,u}(x)- \int_{B_2^c} \frac{(\tau u)(y)}{|x-y|^{n+2s}}\,dy\\
&=f_{1,u}(x)+f_{2,u}(x)- \int_{B_2^c} \frac{(\tau u)(y)}{|y|^{n+2s}\,|x_y-y_y|^{n+2s}}\,dy
,\end{split}
\end{equation}
where the short notation $x_y:=x/|y|$ and $y_y:=y/|y|$ has been exploited.

Now, for any $e\in\partial B_1$ and any $z\in B_{1/2}$, we set
$$ g_e(z):= |z-e|^{-n-2s}.$$
We consider a Taylor expansion of $g_e$ in the vicinity of the origin,
and we write
\begin{equation}\label{PSID1} g_e(z) = \sum_{|\alpha|\le k-1} c_{\alpha,e} \,z^\alpha
+ \sum_{|\alpha|=k} \varrho_\alpha (e,z)\,z^\alpha,\end{equation}
with
\begin{equation}\label{PSI BUONA 2}
\sup_{ {|\alpha|\le k-1}\atop{e\in\partial B_1} } c_{\alpha,e} + 
\sup_{
{|\alpha|=k}\atop{ {e\in\partial B_1}\atop{ z\in B_{1/2} } } } |\partial_z^\gamma \varrho_\alpha(e,z)|
\le C_\gamma,\end{equation}
for some $C_\gamma>0$, which depends only on $n$, $s$ and $\gamma\in\N^n$.

As a consequence, 
we have
\begin{equation} \label{POLx2}
\begin{split}
\frac{(\tau u)(y)}{|y|^{n+2s}\,
\left|x_y-y_y\right|^{n+2s}} \,&=
\frac{(\tau u)(y)}{|y|^{n+2s} }\, g_{y_y} (x_y) \\ &=\frac{
(\tau u)(y)}{|y|^{n+2s} }\, 
\left[
\sum_{|\alpha|\le k-1} c_{\alpha,y_y} \,x_y^\alpha
+ \sum_{|\alpha|=k}\varrho_\alpha(y_y,x_y)\,x_y^\alpha
\right] 
\\&= \frac{(\tau u)(y)}{|y|^{n+2s} }\,
\left[
\sum_{|\alpha|\le k-1} \frac{c_{\alpha,y_y} \,x^\alpha}{|y|^{|\alpha|}}
+ \sum_{|\alpha|=k}\frac{\varrho_\alpha(y_y,x_y)\,x^\alpha}{|y|^k}
\right]
. \end{split}\end{equation}
Thus, we set
$$ \kappa_{\tau,\alpha}:=-\int_{B_2^c}
\frac{(\tau u)(y)}{|y|^{n+2s+|\alpha|} }\,
c_{\alpha,y_y} \,dy $$
and we consider the polynomial of degree at most $k-1$
$$ P_{u,\tau}(x):=
\sum_{|\alpha|\le k-1} 
\kappa_{\tau,\alpha}\,x^\alpha.$$
We also define
\begin{equation}\label{PSID2}
\psi(x,y):=-
\sum_{|\alpha|=k} {\varrho_\alpha(y_y,x_y)\,x^\alpha}.\end{equation}
Notice that \eqref{PSI BUONA} follows from \eqref{PSI BUONA 2}. Also,
with this notation, we 
deduce from \eqref{POLx2} that
$$ \int_{B_2^c} \frac{(\tau u)(y)
}{|y|^{n+2s}\,
\left|x_y-y_y\right|^{n+2s}} = -P_{u,\tau}(x)-f^\star_{u,\tau}(x).$$
This and \eqref{POLx1} imply \eqref{11u8u2139}.
\end{proof}

Then, we have the following consequence of Theorem \ref{POLLI}:

\begin{corollary}\label{POLLI2}
Let $s\in(0,1)$, $k\in\N$ and $u:\R^n\to\R$.

Assume that $u\in C^\alpha_{\rm loc}(B_1)$ for some $\alpha>2s$ and
\begin{equation}\label{INTE}
\int_{\R^n} \frac{|u(y)|}{1+|y|^{n+2s+k}}\,dy < +\infty
.\end{equation}
Let $\tau_R:\R^n\to[0,1]$ be supported in~$B_R$ and such that
\begin{equation}\label{taur}
\lim_{R\to+\infty} \tau_R =1 \quad{\mbox{ a.e. in }}\R^n.
\end{equation}
Then, there exist a function $f_u:\R^n\to\R$,
and a family of polynomials $P_{u,\tau_R}$, which have degree at most $k-1$,
such that,
for any $x\in B_1$, it holds that
\begin{equation}\label{POLx5} \lim_{R\to+\infty} \Big[
(-\Delta)^s (\tau_R u)(x)- P_{u,\tau_R}(x) \Big]
=f_u(x).\end{equation}
More precisely, we have that
\begin{equation}\label{FU1}
f_{u}=f_{1,u}+f_{2,u}+f_{3,u},\end{equation}
where $f_{1,u}$ and $f_{2,u}$ are as in \eqref{f3dif} and
\begin{equation}\label{F3U}
f_{3,u}(x):= 
\int_{B_2^c} \frac{u(y)\;\psi(x,y)}{|y|^{n+2s+k}}\,dy,\end{equation}
with $\psi $ satisfying \eqref{PSI BUONA}.
\end{corollary}

\begin{proof} The idea of the proof is to use Theorem \ref{POLLI}
with $\tau:=\tau_R$ for any fixed $R$, and then send $R\to+\infty$.
Indeed, by \eqref{PSI BUONA}, for any $x\in B_1$ and $y\in B_2^c$,
$$ \frac{(\tau_R u)(y)\;\psi(x,y)}{|y|^{n+2s+k}} \le
\frac{C\,|u(y)|}{|y|^{n+2s+k}},$$
for some $C>0$, and the latter function of $y$ lies in $L^1(B_2^c)$,
thanks to \eqref{INTE}.

Consequently, by \eqref{taur} and the Dominated Convergence Theorem,
$$ \lim_{R\to+\infty}f^\star_{u,\tau_R}(x)
=\lim_{R\to+\infty}\int_{B_2^c} \frac{(\tau_R u)(y)\;\psi(x,y)}{|y|^{n+2s+k}}\,dy=
\int_{B_2^c} \frac{u(y)\;\psi(x,y)}{|y|^{n+2s+k}}\,dy=f_{3,u}(x).$$
Then, \eqref{POLx5} follows by taking the limit in \eqref{11u8u2139}.
\end{proof}

\begin{remark}\label{RE0}
{\rm We stress that, in view of
\eqref{PSID1} and \eqref{PSID2}, the function $\psi$ does not depend on $u$
and thus the quantity
in \eqref{PSI BUONA} is universal.
}\end{remark}

\begin{remark}\label{RE1}
{\rm It is interesting to notice that, from \eqref{FU0} and \eqref{FU1},
$$ f_{u,\tau}= f_u-f_{3,u}+f^\star_{u,\tau}.$$
}\end{remark}

\begin{remark}\label{RE2}
{\rm {F}rom Definition \ref{DE:k} and
Corollary \ref{POLLI2} (used here with~$\tau_R:=\chi_R$,
in the notation of~\eqref{9d01}), we can write $(-\Delta)^s u\ugu f_u$ in $B_1$,
for any $u\in C^\alpha_{\rm loc}(B_1)$ (for some $\alpha>2s$) that satisfies
the weak growth condition in \eqref{INTE}.
}\end{remark}

\begin{remark}\label{RE3}
{\rm {F}rom Corollary \ref{POLLI2} and Remark \ref{RE2},
it also follows that, for any $u\in C^\alpha_{\rm loc}(B_1)$
(for some $\alpha>2s$), the family of cutoffs~$\chi_R$ used in Definition~\ref{DE:k}
can be replaced by another family of cutoffs~$\tau_R$, without changing the
explicit expression of~$f_u$.
}\end{remark}

Another useful consequence of Theorem \ref{POLLI} is that
the pointwise convergence of $f_R$ in Definition \ref{DE:k} can be strengthen
according to the following result:

\begin{corollary}\label{iodjvfJJJAAK}
Let~$k\in\N$, $u$ and $f_R$ be as in Definition \ref{DE:k}. Then,
for any $m\ge 0$,
if $R'>R$ we have that
\begin{equation}\label{7181sjjsjs0019:1}
\inf \|D^m (f_{R'}-f_R-P)\|_{L^\infty (B_1)}\le C\,\int_{B_R^c}
\frac{|u(y)|}{|y|^{n+2s+k}}\,dy,\end{equation}
with $C>0$ only depending on $n$, $s$ and $m$, where the $\inf$ is taken over
all the polynomials $P$ with degree at most~$k-1$.
\end{corollary}

\begin{proof} We define $v:= (1-\chi_2)u$. Obviously, $v=0$ in $B_2$ and $|v|\le |u|$, so
\begin{equation}\label{78sghj199191sjdh}
{\mbox{$v\in C^\alpha_{\rm loc}(B_1)$ for some $\alpha>2s$
and $J_{v,k}\le J_{u,k}<+\infty$.}}
\end{equation}
Moreover, if $R>2$,
$$ (\chi_R-\chi_2) u = (\chi_R-\chi_2) v.$$
Hence,
from \eqref{98e28jkgf12345s},
\begin{equation} \label{0odigjvtwdeiufgvjkgo82345}
(-\Delta)^s ((\chi_R-\chi_2) v)=
(-\Delta)^s ((\chi_R-\chi_2) u) = f_R -f_2 + P_R-P_2=f_R-f_2 +\tilde P_R,\end{equation}
where $\tilde P_R:=P_R-P_2$ is a polynomial
of degree at most $k-1$, and the equation holds in $B_1$ in the sense of viscosity.

On the other hand, \eqref{78sghj199191sjdh} allows us to use
Theorem \ref{POLLI} on the function $v$ (with $\tau:=\chi_R$ and $\tau:=\chi_2$).
We thus obtain that
\begin{eqnarray*}
&& (-\Delta)^s ((\chi_R-\chi_2) v)=P_{v,\chi_R}-P_{v,\chi_2}+f_{v,\chi_R}-f_{v,\chi_2}\\
&&\qquad\quad=
\bar P_{v,\chi_R}+(f_{1,v}+f_{2,v}+f^\star_{v,\chi_R})-
(f_{1,v}+f_{2,v}+f^\star_{v,\chi_2})=
\bar P_{v,\chi_R}+f^\star_{v,\chi_R}-f^\star_{v,\chi_2}\\
&&\qquad\quad=\bar P_{v,\chi_R} +
\int_{B_R\setminus B_2} \frac{u(y)\,\psi(x,y)}{|y|^{n+2s+k}}\,dy
\end{eqnarray*}
in $B_1$ in the viscosity sense, where $\bar P_{v,\chi_R}:=
P_{v,\chi_R}-P_{v,\chi_2}$ is a polynomial
of degree at most $k-1$.

Comparing this identity with \eqref{0odigjvtwdeiufgvjkgo82345}, we obtain that in $B_1$
$$ f_R = f_2+P_R^\star + \int_{B_R\setminus B_2} \frac{u(y)\,\psi(x,y)}{|y|^{n+2s+k}}\,dy,$$
where $P_R^\star:=\bar P_{v,\chi_R}-\tilde P_R$
is a polynomial
of degree at most $k-1$.

Therefore, for any $m\ge 0$
and any large $R'>R$,
\begin{equation}\label{9873irtugefjhdvdghaiugdgf234}
\| D^m(
f_{R'}-P_{R'}^\star-f_R+P_R^\star)\|_{L^\infty(B_1)}
= \|D^m \Psi_{R',R}\|_{L^\infty(B_1)},\end{equation}
where
$$ \Psi_{R',R}(x):=\int_{B_{R'}\setminus B_R} \frac{u(y)\,\psi(x,y)}{|y|^{n+2s+k}}\,dy.$$
{F}rom \eqref{G:k}
and \eqref{PSI BUONA}, we know that
$$ \|\Psi_{R',R}\|_{C^m(B_1)}\le C\,
\int_{B_{R'}\setminus B_R} \frac{|u(y)|}{|y|^{n+2s+k}}\,dy
\le C\,\int_{B_R^c} \frac{|u(y)|}{|y|^{n+2s+k}}\,dy,$$
for some $C>0$ possibly depending on $m$.
This and \eqref{9873irtugefjhdvdghaiugdgf234} imply that
\begin{equation}\label{x1526374s9873irtugefjhdvdghaiugdgf234}
\| D^m(
f_{R'}-P_{R'}^\star-f_R+P_R^\star)\|_{L^\infty(B_1)}\le
C\,\int_{B_R^c} \frac{|u(y)|}{|y|^{n+2s+k}}\,dy,
\end{equation}
which gives~\eqref{7181sjjsjs0019:1}.
\end{proof}

As a consequence of Corollary \ref{iodjvfJJJAAK}, we have the following
consistency result when $k=0$:

\begin{corollary}\label{8IJAJS2wedfv1234}
Let $u:\R^n\to\R$ be bounded and continuous in $B_1$ and such that
\begin{equation}\label{6384rytighk97324} 
\int_{\R^n} \frac{|u(y)|}{1+|y|^{n+2s}}\,dy <+\infty.\end{equation}
Let $f$ be bounded and continuous in $B_1$.

Then
$$ {\mbox{$(-\Delta )^s u=f$ in $B_1$ in the viscosity sense}}$$
is equivalent to
$$ {\mbox{$(-\Delta )^s u\;{\stackrel{0}{=}}\;f$ in $B_1$ in the sense of Definition \ref{DE:k}.}}$$
\end{corollary}

\begin{proof} We take cutoffs as in \eqref{9d01}.
Suppose first that
$(-\Delta )^s u=f$ in $B_1$ in the viscosity sense. Then, for $R>10$,
\begin{equation}\label{8900123}
(-\Delta )^s(\chi_{R/2} u)=f+\int_{\R^n} \frac{(1-\chi_{R/2}(y))\,u(y)}{|x-y|^{n+2s}}\,dy
\end{equation}
in $B_1$ in the viscosity sense.
Now, we set
$$ w:=(\chi_R-\chi_{R/2})u.$$
Notice that $w$ vanishes outside $B_R$, hence 
$$ \chi_R w = w.$$
Also, $w=0$ in $B_2$, so we can exploit Theorem \ref{POLLI}
to $w$ with $k=0$ and get that, for any $x\in B_1$,
\begin{equation}\label{8900123-2}
\begin{split}
&(-\Delta)^s \big( (\chi_R-\chi_{R/2})u \big) (x)=
(-\Delta)^s w(x) =
(-\Delta)^s(\chi_R w ) (x)
\\ &\qquad=f_{1,w}(x)+f_{2,w}(x)
+f^\star_{w,\chi_R}(x) =\int_{B_R\setminus B_2} \frac{w(y)\,\psi(x,y)}{|y|^{n+2s}}\,dy
=\int_{B_R\setminus B_{R/2}} \frac{u(y)\,\psi(x,y)}{|y|^{n+2s}}\,dy.
\end{split}\end{equation}
Since $w$ is smooth in $B_1$, this identity also holds in the viscosity sense.
Hence, from \eqref{8900123} and \eqref{8900123-2}, we find that
\begin{eqnarray*}
(-\Delta )^s(\chi_{R} u) &=&
(-\Delta)^s \big( (\chi_R-\chi_{R/2})u \big) +
(-\Delta )^s(\chi_{R/2} u)\\ &=&
\int_{B_R\setminus B_{R/2}} \frac{u(y)\,\psi(x,y)}{|y|^{n+2s}}\,dy
+
f+\int_{\R^n} \frac{(1-\chi_{R/2}(y))\,u(y)}{|x-y|^{n+2s}}\,dy
\\ &=:& f_R\end{eqnarray*}
in $B_1$, in the sense of viscosity.
We remark that $f_R\to f$ in $B_1$ as~$R\to+\infty$, thanks to \eqref{PSI BUONA}
and \eqref{6384rytighk97324}. Hence, we are in the setting of
Definition~\ref{DE:k} (here with $k=0$ and $P_R:=0$),
and so we conclude that
$(-\Delta )^s u
{\stackrel{0}{=}}f$ in $B_1$, as desired.

Viceversa, we now suppose that $(-\Delta )^s u
{\stackrel{0}{=}}f$ in $B_1$. {F}rom
Definition~\ref{DE:k} and the fact that $k=0$, we have that $P_R$
is always zero, and so we can write that
$(-\Delta )^s (\chi_R u)=f_R$ in $B_1$ in the sense of viscosity, 
with $f_R\to f$ in $B_1$ as~$R\to+\infty$.

We observe that $\chi_R u$ approaches $u$ locally uniformly in $\R^n$.
Also, we can use here Corollary \ref{iodjvfJJJAAK}: since
in this case~$k=0$, we have that~\eqref{7181sjjsjs0019:1} reduces to
$$ \|D^m(f_{R'}-f_R)\|_{L^\infty (B_1)}\le C\,\int_{B_R^c}
\frac{|u(y)|}{|y|^{n+2s}}\,dy,$$
for any~$m\ge0$. In particular, taking~$m=0$ and sending~$R'\to+\infty$,
we obtain that, for any~$x\in B_1$,
$$ |f(x)-f_R(x)|=
\lim_{R'\to+\infty}|f_{R'}(x)-f_R(x)|
\le \lim_{R'\to+\infty} 
\| f_{R'}-f_R\|_{L^\infty (B_1)}\le C\,\int_{B_R^c}
\frac{|u(y)|}{|y|^{n+2s}}\,dy.$$
As a consequence, we have that
$f_R$ converges to $f$ uniformly in $B_1$ as $R\to+\infty$.

{F}rom this, we can exploit
Lemma~5 in~\cite{MR2781586} and conclude that $(-\Delta)^s u=f$
in the viscosity sense in $B_1$, as desired.
\end{proof}

Another consistency result is that if $(-\Delta)^s u\ugu f$
and $u$ has growth at infinity better than the one required by
Definition \ref{DE:k}, then it satisfies the same equation
``in a better class, up to
the invariance allowed by Definition \ref{DE:k}''. The precise result is as follows:

\begin{lemma}\label{3401oAJ891010}
Let the setting of Definition \ref{DE:k} hold true and let~$(-\Delta)^s u\ugu f$
in~$B_1$. Suppose that
\begin{equation}\label{G:inv j}
\int_{\R^n} \frac{|u(y)|}{1+|y|^{n+2s+j}}\,dy < +\infty
\end{equation}
for some~$j\in\N$, with $j\le k$. Then, there exist
a function $\bar f$ and a polynomial $P$ of degree at most $k-1$,
such that~$\bar f=f+P$ and~$
(-\Delta)^s u\;{\stackrel{j}{=}}\;\bar f$ in~$B_1$.
\end{lemma}

\begin{proof} Let~$v:=(1-\chi_4)u$ and~$w:=\chi_4 u$. Of course,
$v$ is zero (and thus smooth) in $B_1$ and, from \eqref{G:inv j},
we have that
$$ \int_{\R^n} \frac{|v(y)|}{1+|y|^{n+2s+j}}\,dy < +\infty.$$
So, we can apply Remark~\ref{RE2} with~$k$ replaced with~$j$
and find that
\begin{equation*}
(-\Delta)^s v\; {\stackrel{j}{=}} \;f_v=\int_{B_4^c}\frac{u(y)\,
\psi(x,y)}{|y|^{n+2s+j}}\,dy
,\end{equation*}
thanks to~\eqref{FU1}. By definition, this means that
\begin{equation}\label{p01-01}
(-\Delta)^s (\chi_R v)=
\int_{B_4^c}\frac{u(y)\,\psi(x,y)}{|y|^{n+2s+j}}\,dy +\varphi_R+Q_R
,\end{equation}
in the viscosity sense in~$B_1$,
for some $\varphi_R$ such that~$\varphi_R\to0$ in~$B_1$ as~$R\to+\infty$
and a polynomial~$Q_R$ of degree at most~$j-1$.

On the other hand, from Definition \ref{DE:k},
we have that
\begin{equation}\label{p01-02}
(-\Delta)^s (\chi_R u)= f+\phi_R+P_R 
,\end{equation}
in the viscosity sense in~$B_1$,
for some $\phi_R$ such that~$\phi_R\to0$ in~$B_1$ as~$R\to+\infty$
and a polynomial~$P_R$ of degree less than or equal to~$k-1$.

By subtracting~\eqref{p01-01} to~\eqref{p01-02}, we obtain
$$ f+\phi_R+P_R-
\int_{B_4^c}\frac{u(y)\,\psi(x,y)}{|y|^{n+2s+j}} \,dy -\varphi_R-Q_R
=(-\Delta)^s (\chi_R (u-v)) = (-\Delta)^s(\chi_R w)=(-\Delta)^s(\chi_4 w)$$
in the viscosity sense in~$B_1$.
This says that the following limit exists:
$$ \lim_{R\to+\infty}\left(
\phi_R+P_R-\varphi_R-Q_R\right) ,$$
which in turn boils down to the existence of the limit
$$ \lim_{R\to+\infty} (P_R-Q_R).$$
As a consequence, from Lemma~\ref{9uids2345hv}, we know that
$$ \lim_{R\to+\infty} (P_R-Q_R) =P,$$
for some polynomial~$P$ of degree at most~$k-1$.
That is, we take~$\bar f:=f+P$ and~$\Phi_R:=\phi_R+P_R-Q_R-P$,
and we see that~$\Phi_R\to0$ as~$R\to+\infty$ and, from~\eqref{p01-02},
\begin{equation*}
(-\Delta)^s (\chi_R u)= \bar f+\Phi_R+Q_R
\end{equation*}
in~$B_1$, in the viscosity sense. Since the degree of~$Q_R$ is at most~$j-1$,
this says that~$
(-\Delta)^s u\;{\stackrel{j}{=}}\;\bar f$ in~$B_1$, as desired.
\end{proof}

For us, Lemma~\ref{3401oAJ891010} is useful since it allows
to take fixed cutoffs in
Definition \ref{DE:k} and reduce to the case of the standard fractional Laplacian,
as formalized by the following result:

\begin{corollary}\label{2345923847hffhhfhfhfsksk}
Let the setting of Definition \ref{DE:k} hold true and let~$(-\Delta)^s u\ugu f$
in~$B_1$. Let also~$\rho\ge1$ and~$w:=\chi_\rho u$. Then, there exists a polynomial~$P$ of degree
at most~$k-1$ such that
\begin{equation}\label{7654rdjf1gkfjhdgsfj}
(-\Delta)^s w = \bar f+
\int_{B_2\cap B_\rho^c} \frac{u(y)}{|x-y|^{n+2s}}\,dy
-\int_{B_2^c\cap B_\rho^c}\frac{u(y)\,\psi(x,y)}{|y|^{n+2s+k}}\,dy
\end{equation}
in~$B_1$ in the sense of viscosity,
where~$\bar f:=f+P$.
\end{corollary}

\begin{proof} {F}rom Definition \ref{DE:k}, we can write, in~$B_1$
and in the viscosity sense,
$$ (-\Delta)^s (\chi_Ru)= f+\phi_R+P_R,$$
where~$P_R$ is a polynomial with degree at most~$k-1$ and~$\phi_R\to0$ as~$R\to+\infty$.
We also set~$v:=(1-\chi_\rho)u$. Notice that~$v=0$ in~$B_1$.
We can apply Remark~\ref{RE2} to~$v$
and find that, in~$B_1$,
\begin{eqnarray*}
&&(-\Delta)^s v\;{\stackrel{k}{=}}\; f_v=
-\int_{B_2} \frac{v(y)}{|x-y|^{n+2s}}\,dy+
\int_{B_2^c}\frac{v(y)\,
\psi(x,y)}{|y|^{n+2s+k}}\,dy\\
&&\qquad\qquad=
-\int_{B_2\setminus B_\rho} \frac{u(y)}{|x-y|^{n+2s}}\,dy+
\int_{B_2^c\cap B_\rho^c}\frac{u(y)\,
\psi(x,y)}{|y|^{n+2s+k}}\,dy
,\end{eqnarray*}
where we used the obvious
notation~$B_2\setminus B_\rho=\varnothing$ if~$\rho\ge2$.

That is, in~$B_1$
and in the viscosity sense,
$$ (-\Delta)^s(\chi_R v)= -\int_{B_2\setminus B_\rho} \frac{u(y)}{|x-y|^{n+2s}}\,dy+
\int_{B_2^c\cap B_\rho^c}\frac{u(y)\,\psi(x,y)}{|y|^{n+2s+k}}\,dy+\tilde\phi_R+\tilde P_R,$$
where~$\tilde P_R$ is a polynomial with degree at most~$k-1$ and~$\tilde\phi_R\to0$ as~$R\to+\infty$.
Consequently,
\begin{eqnarray*} (-\Delta)^s (\chi_R w)&= &(-\Delta)^s (\chi_R u)-(-\Delta)^s (\chi_R v)\\ &=&
f
+\int_{B_2\setminus B_\rho} \frac{u(y)}{|x-y|^{n+2s}}\,dy
-\int_{B_2^c\cap B_\rho^c}\frac{u(y)\,\psi(x,y)}{|y|^{n+2s+k}}\,dy
+\phi_R
-\tilde\phi_R+P_R-\tilde P_R,\end{eqnarray*}
which means that, in~$B_1$,
\begin{equation}\label{1ertgsdfg-0734r}
(-\Delta)^s w\;{\stackrel{k}{=}} \;
f+
\int_{B_2\setminus B_\rho} \frac{u(y)}{|x-y|^{n+2s}}\,dy
-\int_{B_2^c\cap B_\rho^c}\frac{u(y)\,\psi(x,y)}{|y|^{n+2s+k}}\,dy.
\end{equation}
We remark that~$w$ is a compactly supported function,
hence~\eqref{G:inv j} holds true for~$j=0$. 
Thus, from~\eqref{1ertgsdfg-0734r} and
Lemma~\ref{3401oAJ891010}, we obtain that
\begin{equation*}
(-\Delta)^s w\;{\stackrel{0}{=}}\;
f+\int_{B_2\setminus B_\rho} \frac{u(y)}{|x-y|^{n+2s}}\,dy
-\int_{B_2^c\cap B_\rho^c}\frac{u(y)\,\psi(x,y)}{|y|^{n+2s+k}}\,dy+P
\end{equation*}
in~$B_1$, where~$P$ is a polynomial of degree at most~$k-1$. 
This and Corollary~\ref{8IJAJS2wedfv1234}
imply~\eqref{7654rdjf1gkfjhdgsfj}, as desired.
\end{proof}

It is interesting to point out that, in the setting of
Definition \ref{DE:k}, the functions $f_R$ and~$f$ are not necessarily smooth,
hence one cannot deduce from Corollary~\ref{iodjvfJJJAAK}
that ``$f_R$ converges to~$f$ in~$C^m(B_1)$''.
Also, in principle, one cannot get rid of
the additional polynomials in Corollary~\ref{iodjvfJJJAAK},
since they come from the polynomial invariance of Definition \ref{DE:k}.

In spite of this, it is possible to give a sharper version of
Corollary~\ref{iodjvfJJJAAK}, by
introducing a notion of
``optimal representative'' 
for the functions $f_R$
in Definition \ref{DE:k}, which, in principle, are only ``well defined
up to polynomials of degree $k-1$''. This will be accomplished
by looking at ``projection over the orthogonal space to polynomials''.
Namely, for any $g\in L^2(B_1)$ we
look at the minimum of $\| g+P\|_{L^2(B_1)}$
among all the polynomials~$P$ of degree at most $k-1$.
We remark that such minimum exists, since the space of polynomials
is finite dimensional, and it is unique, due to the strict convexity of the norm,
so we define the minimizing polynomial as $P^\sharp_g$. 

Then we set
\begin{equation}\label{78yis77791883} g^\sharp := g+P^\sharp_g.\end{equation}
In this setting, we have:

\begin{lemma}\label{09sc5tgdwefguv}
Let the setting of Definition \ref{DE:k} hold true.
Then
\begin{equation}\label{896805678uhgfi}
\lim_{R\to+\infty} f^\sharp_R = f^\sharp\end{equation}
a.e. in $B_1$.
Also, 
\begin{equation}\label{896805678uhgfi:33}
{\mbox{$f^\sharp-f$ is a polynomial of degree at most $k-1$.}}\end{equation}
Furthermore, 
for any $m\ge0$, we have that 
\begin{equation}\label{896805678uhgfi:2}
\|f^\sharp-f_R^\sharp\|_{C^m(B_1)}\le C\,\int_{B_R^c}
\frac{|u(y)|}{|y|^{n+2s+k}}\,dy,\end{equation}
with $C>0$ only depending on $n$, $s$ and $m$.
\end{lemma}

\begin{proof} We set
$$ \nu_R:=\int_{B_R^c}
\frac{|u(y)|}{|y|^{n+2s+k}}\,dy.$$
We claim that
\begin{equation}\label{u98AHSHJ81923bjf12ru}
\| f_R^\sharp-f^\sharp\|_{L^2(B_1)}\le C\nu_R.
\end{equation}
For this, we observe that, for any function $g$ and any polynomial $P$
with degree at most $k-1$, we have that 
\begin{equation}\label{9oxcs8uhdhgdgg1239}
(g+P)^\sharp=g^\sharp.\end{equation}
Also, from the minimizing property of $P_g^\sharp$
we see that $g^\sharp$ is orthogonal in $L^2(B_1)$ to all the polynomials of
degree at most $k-1$ and therefore, for any functions $g$ and $h$, we have that
\begin{equation}\label{8iwydh678we8difjvbuiyedqwihoich}
\begin{split}&\| g-h\|_{L^2(B_1)}^2 =
\| (g^\sharp-h^\sharp)-(P_g^\sharp-P_h^\sharp)\|^2_{L^2(B_1)} \\&\qquad
=\| g^\sharp-h^\sharp\|^2_{L^2(B_1)}+\|P_g^\sharp-P_h^\sharp\|^2_{L^2(B_1)}
\ge \| g^\sharp-h^\sharp\|^2_{L^2(B_1)}.\end{split}\end{equation}
Now, for~$R'>R$, let $P_{R',R}$ be such that
$$ \|f_{R'}-f_R-P_{R',R}\|_{L^\infty (B_1)} = 
\min \|f_{R'}-f_R-P\|_{L^\infty(B_1)},$$
where the minimization is meant over all the polynomials $P$ of degree at most $k-1$.
{F}rom \eqref{7181sjjsjs0019:1} (used here with $m=0$), we know that
$$ \|f_{R'}-f_R-P_{R',R}\|_{L^2(B_1)}\le \|f_{R'}-f_R-P_{R',R}\|_{L^\infty (B_1)}\le
C\nu_R.$$
Hence, in view of \eqref{8iwydh678we8difjvbuiyedqwihoich}, we have that
$$ \|f_{R'}^\sharp-(f_R+P_{R',R})^\sharp\|_{L^2(B_1)} \le C\nu_R.$$
This and \eqref{9oxcs8uhdhgdgg1239} give that
\begin{equation}\label{3.32bis}
\|f_{R'}^\sharp-f_R^\sharp\|_{L^2(B_1)} \le C\nu_R.\end{equation}
Thus, we can pass to the limit as $R'\to+\infty$ and use Fatou's Lemma
to obtain \eqref{u98AHSHJ81923bjf12ru}, as desired.

Notice that, from \eqref{u98AHSHJ81923bjf12ru}, up to a subsequence
we obtain~\eqref{896805678uhgfi}.

Then, from~\eqref{896805678uhgfi}, we have that, a.e. in $B_1$,
$$ f^\sharp-f =
\lim_{R\to+\infty} (f^\sharp_R -f_R)= \lim_{R\to+\infty} P^\sharp_{f_R}.$$
This and Lemma~\ref{9uids2345hv} imply that $f^\sharp-f$
is a polynomial of degree at most $k-1$,
and this proves~\eqref{896805678uhgfi:33}.

Then, in view of \eqref{0dfhvnoooP} and~\eqref{896805678uhgfi:33},
we have that $(-\Delta)^s u\ugu f^\sharp$.

This fact and \eqref{896805678uhgfi}
give that
we can use Corollary \ref{iodjvfJJJAAK} with the function $f^\sharp_R$:
in this way, we fix
\begin{equation}\label{789iJJJSHHSIIIq34}
m\ge k+n+2 \end{equation}
and we have from \eqref{7181sjjsjs0019:1} that, for any $R'>R$,
\begin{equation}\label{67yc7qw019tgvkjbdsiqwyc}
\|D^m (f_{R'}^\sharp-f_R^\sharp)\|_{L^\infty (B_1)}\le C\nu_R
.\end{equation}
Now we recall the Gagliardo-Nirenberg
Interpolation Inequality (see e.g. pages~125-126 in \cite{MR0109940}), namely,
for any~$i\le m\in\N$,
$$ \|D^i \varphi\|_{{L^{2}(B_1)}}\leq C\,\Big(
\|\varphi\|_{{L^{2}(B_1)}}^{{1-\frac{i}m}}\,
\|D^m \varphi\|_{{L^{2}(B_1)}}^{{\frac{i}m}}+
\| \varphi\|_{{L^{2}({B_1})}}\Big),$$
for some~$C>0$. Taking $\varphi:=f_{R'}^\sharp-f_R^\sharp$
and using \eqref{3.32bis}
and \eqref{67yc7qw019tgvkjbdsiqwyc},
we conclude that, for $R'>R$ large enough,
\begin{equation} \label{0osxjc41253678943827}
\|D^i (f_{R'}^\sharp-f_R^\sharp)\|_{{L^{2}(B_1)}}\leq C\nu_R,\end{equation}
for any $i\in\{0,\dots,m\}$, up to renaming $C>0$.

Now, since we do not know if~$f_R^\sharp$ is sufficiently smooth,
we perform a technical argument to take limits.
Namely, we set
\begin{equation}\label{623192837} \xi_{R}:=f^\sharp_{R}-f^\sharp .\end{equation}
{F}rom \eqref{u98AHSHJ81923bjf12ru}
and \eqref{0osxjc41253678943827},
for any $\phi\in C^\infty_0(B_1)$ and any $\iota\in\N^n$ with $|\iota|=i
\in\{0,\dots,m\}$, we have that
\begin{eqnarray*}
&& \int_{B_1} D^\iota\xi_R \phi=
(-1)^i \int_{B_1} \xi_R D^\iota\phi=
(-1)^i \int_{B_1} (f^\sharp_{R}-f^\sharp)D^\iota\phi\\
&&\qquad=\lim_{R'\to+\infty} (-1)^i \int_{B_1} (f^\sharp_{R}-f^\sharp_{R'})D^\iota\phi
=\lim_{R'\to+\infty}\int_{B_1} D^\iota(f^\sharp_{R}-f^\sharp_{R'})\phi\\
&&\qquad\le
C\nu_R \int_{B_1} |\phi|\le C\nu_R\,\|\phi\|_{L^2(B_1)}.
\end{eqnarray*}
Then, by the density of~$C^\infty_0(B_1)$ in~$L^2(B_1)$, this inequality holds
for any~$\phi\in L^2(B_1)$, and thus
$$ \|D^\iota\xi_R\|_{L^2(B_1)}=\sup_{0\ne \phi\in L^2(B_1)}
\frac{\displaystyle\int_{B_1} D^\iota\xi_R \phi}{\|\phi\|_{L^2(B_1)}}\le C\nu_R.$$
Accordingly, since this is valid for all~$|\iota|=i\in\{0,\dots,m\}$,
$$ \|\xi_R\|_{{W^{m,2}(B_1)}}\leq C\nu_R,$$
up to renaming constants.

{F}rom this and the Sobolev Inequality, recalling also \eqref{623192837},
it follows\footnote{We use the standard notation for
the integer part of a real number. Namely, given~$\varrho\in\R$,
we denote by~$\lfloor\varrho\rfloor:=\max\{ m\in\Z {\mbox{ s.t. }}m\le\varrho\}$.}
that \begin{equation*}
\| f^\sharp_{R}-f^\sharp
\|_{C^{m'}(B_1)}=
\| \xi_R\|_{C^{m'}(B_1)}\le C\nu_R,\end{equation*}
with $m'= m-\lfloor \frac{n}{2}\rfloor-1$, up to renaming $C>0$.
This is the desired result in~\eqref{896805678uhgfi:2}, up to renaming $m$ in the statement
of Lemma \ref{09sc5tgdwefguv}.
\end{proof}

\section{Schauder estimates
for divergent fractional Laplacians}\label{SECT:3}

This section is devoted to the proof of Theorem~\ref{SCHAUDER}.
For this, we first 
give a uniform bound  for solutions, as stated in the following result:

\begin{lemma}\label{SCk91001}
Let $\beta\in(0,1)$, $m\in\N$, $s\in(0,1)$
and $u:\R^n\to\R$.

Assume that $u$ is continuous in~$B_1$ and
that
\begin{equation*}
J_{u,0}:=\int_{B_{1/2}^c} \frac{|u(y)|}{|y|^{n+2s}}\,dy<+\infty
.\end{equation*}
Suppose that
$$ (-\Delta)^s u = f \qquad{\mbox{ in }}\;B_1$$ in the viscosity sense.
Then
$$ \|u\|_{L^\infty(B_{9/10})}+\|f\|_{C^m(B_{99/100})}\le C\,\Big( [f]_{C^{m+\beta}(B_1)}
+J_{u,0}\Big),$$
for some $C>0$, only depending on $n$, $s$, $m$ and~$\beta$.
\end{lemma}

\begin{proof}
By contradiction, we can suppose that there exist sequences
of functions~$u_j$ and~$f_j$ such that~$(-\Delta)^s u_j = f_j$ in $B_1$,
with
\begin{equation}\label{AssuX245X456X5X}
\Theta_j:=\|u_j\|_{L^\infty(B_{9/10})}+\|f_j\|_{C^m(B_{99/100})}>j\,\Big( [f_j]_{C^{m+\beta}(B_1)}
+J_{u_j,0}\Big).
\end{equation}
We define
$$ \tilde u_j:=\frac{u_j}{\Theta_j}
\quad{\mbox{ and }}\quad\tilde f_j:=\frac{f_j}{\Theta_j}.$$
Then,
\begin{equation}\label{AssuX245X456X5X:2}
(-\Delta)^s \tilde u_j = \tilde f_j\quad{\mbox{ in }}B_1.
\end{equation}
Also,
\begin{equation}\label{AssuX245X456X5X:3}
\|\tilde f_j\|_{C^m(B_{99/100})}=\frac{\|f_j\|_{C^m(B_{99/100})}}{\Theta_j}\le1 
\end{equation}
and
\begin{equation}\label{AssuX245X456X5X:4}
[\tilde f_j]_{C^{m+\beta}(B_1)}=
\frac{ [f_j]_{C^{m+\beta}(B_1)} }{ \Theta_j }\le\frac1j,
\end{equation}
due to~\eqref{AssuX245X456X5X}.

In particular, we have that~$\|\tilde f_j\|_{C^{m+\beta}(B_{99/100})}\le2$.
{F}rom this, up to a subsequence, we may suppose that
\begin{equation}\label{09sd888811}
{\mbox{$\tilde f_j$
converges to some~$\tilde f$ in~${C^{m+\beta}(B_{99/100})}$.}}\end{equation}
We also remark that
$$ [\tilde f]_{C^{m+\beta}(B_{99/100})}\le [\tilde f-\tilde f_j]_{C^{m+\beta}(B_{99/100})}+
[\tilde f_j]_{C^{m+\beta}(B_{99/100})}\le
\| \tilde f-\tilde f_j\|_{C^{m+\beta}(B_{99/100})}+\frac1j,$$
which goes to zero as~$j\to+\infty$. This means that~$D^m\tilde f$
is constant in~$B_{99/100}$, hence~$D^{m+1}\tilde f$ vanishes in~$B_{99/100}$
and
\begin{equation}\label{98duwyfideyefuisdgj8991:0}
{\mbox{$\tilde f$ is a polynomial of degree $m$.}}
\end{equation}
Moreover, 
\begin{equation}\label{78879:001}
J_{\tilde u_j,0}= \frac{J_{u_j,0} }{{\Theta_j}}\le\frac1j,
\end{equation}
thanks to~\eqref{AssuX245X456X5X}.

Now, from~\eqref{AssuX245X456X5X:2} and 
Lemma 5.2 in~\cite{2016arXiv160102548C} (see also the remark after it),
we have that
\begin{equation}\label{78879:001:02}
\|\tilde u_j\|_{L^\infty(B_{97/100})} \le
C\,\Big( \|\tilde f_j\|_{L^\infty(B_{99/100})}
+J_{\tilde u_j,0}
\Big),\end{equation}
for some~$C>0$.

Also, by
Proposition~7.1(a)
in~\cite{2016arXiv160102548C}, we have that, for any fixed~$\alpha\in(0,2s)$,
$$ \|\tilde u_j\|_{C^\alpha(B_{9/10})} \le C\,\Big( 
\|\tilde f_j\|_{L^\infty(B_{97/100})}
+\|\tilde u_j\|_{L^\infty(B_{97/100})}+
J_{\tilde u_j,0}
\Big),$$
for some~$C>0$.
Hence, making use of~\eqref{AssuX245X456X5X:3}, \eqref{78879:001}
and~\eqref{78879:001:02}, we conclude that~$ \|\tilde u_j\|_{C^\alpha(B_{9/10})}$
is bounded uniformly in~$j$ and so, up to a subsequence, we may
assume that~$\tilde u_j$ converges to some~$\tilde u$
in~$L^\infty(B_{9/10})$. 

As a matter of fact, from~\eqref{78879:001}, we also know that~$\tilde u_j$
converges to zero a.e. outside~$B_{1/2}$, hence we can extend~$\tilde u$
to be zero outside~$B_{9/10}$ and write that
\begin{equation}\label{8989hdsdhfsjjjsjs12345}
\lim_{j\to+\infty} \| \tilde u_j-\tilde u\|_{L^\infty(B_{9/10})}
+\int_{\R^n} \frac{|\tilde u_j(y)-\tilde u(y)|}{1+|y|^{n+2s}}\,dy=0,
\end{equation}
with
\begin{equation}\label{98duwyfideyefuisdgj8991}
{\mbox{$\tilde u=0$ outside $B_{1/2}$}}.\end{equation}
Hence, exploiting~\eqref{09sd888811}, \eqref{8989hdsdhfsjjjsjs12345}
and
Lemma~5 in~\cite{MR2781586}, we can pass~\eqref{AssuX245X456X5X:2}
to the limit and find that
\begin{equation}\label{0-du9fioehjb}
(-\Delta)^s \tilde u = \tilde f\quad{\mbox{ in }}B_{9/10}.\end{equation}
{F}rom this, \eqref{98duwyfideyefuisdgj8991:0},
\eqref{98duwyfideyefuisdgj8991} and Lemma~\ref{PLO:0},
we obtain that~$\tilde u$ vanishes identically.

This and~\eqref{0-du9fioehjb} give that~$\tilde f=0$ in~$B_{9/10}$
(and in fact, from~\eqref{98duwyfideyefuisdgj8991:0},
we have that~$\tilde f=0$ in $B_{99/100}$). Consequently,
recalling~\eqref{09sd888811} and~\eqref{8989hdsdhfsjjjsjs12345},
\begin{eqnarray*}&& 1=\lim_{j\to+\infty}
\frac{\|u_j\|_{L^\infty(B_{9/10})}+\|f_j\|_{C^m(B_{99/100})} }{\Theta_j}
=\lim_{j\to+\infty}\|\tilde u_j\|_{L^\infty(B_{9/10})}+\|\tilde f_j\|_{C^m(B_{99/100})}
\\ &&\qquad
=\lim_{j\to+\infty}\|\tilde u_j-\tilde u\|_{L^\infty(B_{9/10})}+\|\tilde f_j-\tilde f\|_{C^m(B_{99/100})}
=0,
\end{eqnarray*}
which is, of course, a contradiction.
\end{proof}

To address the Schauder estimates of Theorem~\ref{SCHAUDER},
we now provide a simpler, suboptimal version (this result can be obtained
by a suitable iteration argument from the existing literature, but
we give the precise statement and the details of the proof for the reader's convenience):

\begin{lemma}\label{SUBP}
Let $s\in(0,1)$, $u$ be continuous in $B_1$,
with $u\in L^\infty(\R^n)$, $f:B_1\to \R$ and
suppose that
$$ (-\Delta)^s u = f \qquad{\mbox{ in }}\;B_1$$
in the viscosity sense.
Then, for any $\gamma>0$ for which~$\gamma\not\in\N$ and
$\gamma+2s\not\in\N$,
\begin{equation*}
\|u\|_{C^{\gamma+2s}(B_{1/2})}
\le C\, \Big( \|f\|_{C^\gamma(B_1)}+
\|u\|_{L^\infty(\R^n)}
\Big),
\end{equation*}
for some $C>0$, only depending on $n$, $s$ and~$\gamma$.
\end{lemma}

\begin{proof} We write $\gamma=m+\theta$, with $m:=\lfloor\gamma\rfloor$
and $\theta\in(0,1)$. The proof is by induction over $m$.
When $m=0$, the claim follows from Proposition 7.2(b) in \cite{2016arXiv160102548C}
(or Corollary 3.5 in~\cite{MR3482695}).

Now suppose that the claim is true for some $m\in\N$ and we prove it for $m+1$.
That is, we assume, recursively, that
\begin{equation} \label{IT:1}
\|u\|_{C^{m+\theta+2s}(B_{1/2})}
\le C\, \Big( \|f\|_{C^{m+\theta}(B_1)}+
\|u\|_{L^\infty(\R^n)} \Big),\end{equation}
and we prove the same statement for $m+1$ in the place of $m$ (up to renaming $C$
and possibly resizing balls).

For the sake of simplicity, let us first deal with the case
\begin{equation}\label{IT:2}
\theta+2s>1.
\end{equation}
For this, we take an incremental quotient of order $m+1$,
that is we fix $\omega_1,\dots,\omega_{m+1}\in S^{n-1}$ and we 
let $v:= D_h^{(\omega_1,\dots,\omega_{m+1})} u$ and 
$g:=D_h^{(\omega_1,\dots,\omega_{m+1})}f$ (recall the notation
of finite differences in Appendix~A.
Then, for small $h$, we have that
$$ (-\Delta)^s v = g\qquad{\mbox{ in }}\;B_{9/10}$$
in the viscosity sense. Then, we take $\phi\in C^\infty_0(B_{1/4})$
with $\phi=1$ in $B_{1/8}$ and we define $w:=\phi v$. In this way, we obtain that, for any $x\in B_{1/16}$,
\begin{equation}\label{89aosucdkhoyiwhefkj23456u}
\begin{split}
& (-\Delta)^s w(x) =\int_{\R^n}\frac{v(x)-(\phi v)(y)}{|x-y|^{n+2s}}\,dy
= g(x)+\int_{\R^n}\frac{(1-\phi (y)) v(y)}{|x-y|^{n+2s}}\,dy
\\ &\qquad=
g(x)+\int_{\R^n}\frac{(1-\phi (y)) D_h^{(\omega_1,\dots,\omega_{m+1})}
u(y)}{|x-y|^{n+2s}}\,dy.
\end{split}\end{equation}
Notice that, if, for any $x\in B_{1/16}$, we set
$$ \Psi^{(x)}(y):= \frac{1-\phi (y)}{|x-y|^{n+2s}},$$
we have that $\Psi^{(x)}$ vanishes in $B_{1/8}$
and so $\Psi^{(x)}\in L^1( \R^n)$.
Therefore, by Lemma \ref{ojskdbv63748596},
\begin{eqnarray*}
&&\int_{\R^n}\frac{(1-\phi (y)) D_h^{(\omega_1,\dots,\omega_{m+1})}
u(y)}{|x-y|^{n+2s}}\,dy=
\int_{\R^n}\Psi^{(x)}(y)\,D_h^{(\omega_1,\dots,\omega_{m+1})} u(y)\,dy\\ &&\qquad\qquad=
\int_{\R^n} D_h^{(-\omega_1,\dots,-\omega_{m+1})}
\Psi^{(x)}(y)\, u(y)\,dy=: G(x),\end{eqnarray*}
with $$\| G\|_{C^1(B_{1/16})}\le C\,h^{m+1}\,\|u\|_{L^\infty(\R^n)}.$$
Hence, \eqref{89aosucdkhoyiwhefkj23456u}
gives that $(-\Delta)^s w = H$ in $B_{1/16}$, with $H:=g+G$ and, by
Lemma~\ref{DERIVE},
we have
$$ \| H\|_{C^\theta(B_{1/16})} \le \| g\|_{C^\theta(B_{1/16})}
+\| G\|_{C^\theta(B_{1/16})}\le
C\,h^{m+1}\,\Big(\|f\|_{C^{m+1+\theta}(B_{1/8})}+
\|u\|_{L^\infty(\R^n)}\Big).$$
That is, using the claim with $m=0$ and once more
Lemma~\ref{DERIVE},
\begin{eqnarray*}
\|D_h^{(\omega_1,\dots,\omega_{m+1})} u\|_{C^{\theta+2s}(B_{1/32})}&=&
\|w\|_{C^{\theta+2s}(B_{1/32})}\\
&\le& C\, \Big( \|H\|_{C^{\theta}(B_{1/16})}+
\|w\|_{L^\infty(\R^n)} \Big)\\
&\le& Ch^{m+1}\, \Big( 
\|f\|_{C^{m+1+\theta}(B_{1/8})}+
\|u\|_{L^\infty(\R^n)}+\| u\|_{C^{m+1}(B_{1/4})}
\Big).\end{eqnarray*}
Dividing by $h^{m+1}$, sending $h\to0$ and recalling again
Lemma \ref{DERIVE},
we thus find that
$$ \|D^{m+1} u
\|_{C^{\theta+2s}(B_{1/32})}\le
C\,\Big(\|f\|_{C^{m+1+\theta}(B_{1/8})}+
\|u\|_{L^\infty(\R^n)}+\| u\|_{C^{m+1}(B_{1/4})}\Big).$$
This, together with \eqref{IT:1} and~\eqref{IT:2}, gives that
$$ \|u\|_{C^{m+1+\theta+2s}(B_{1/32})}
\le C\, \Big( \|f\|_{C^{m+1+\theta}(B_1)}+
\|u\|_{L^\infty(\R^n)} \Big),$$
up to renaming $C>0$, which is the iterative version of \eqref{IT:1}
(up to renaming constants and resizing balls), as desired.

If, on the other hand, the condition in \eqref{IT:2} does not hold,
i.e. $\theta+2s\in(0,1)$, then the previous proof must be done step by step,
namely, one takes $\bar N\in\N$ so large that $\bar\alpha:= 1/\bar N< \theta+2s$.
Then one considers the functions
\begin{eqnarray*}&& v(x):= D_h^{(\omega_1,\dots,\omega_{m})} u(x+\omega_{m+1})-
D_h^{(\omega_1,\dots,\omega_{m})} u(x)\\
{\mbox{and }}
&& g(x):= D_h^{(\omega_1,\dots,\omega_{m})} f(x+\omega_{m+1})-
D_h^{(\omega_1,\dots,\omega_{m})} f(x).\end{eqnarray*}
Then, the argument above would give a bound like
$$ \|u\|_{C^{m+\bar\alpha+\theta+2s}(B_{1/32})}
\le C\, \Big( \|f\|_{C^{m+\bar\alpha+\theta}(B_1)}+
\|u\|_{L^\infty(\R^n)} \Big).$$
Hence, one repeats this argument over and over to get
$$ \|u\|_{C^{m+j\bar\alpha+\theta+2s}(B_{r_j})}
\le C\, \Big( \|f\|_{C^{m+j\bar\alpha+\theta}(B_1)}+
\|u\|_{L^\infty(\R^n)} \Big)$$
for every $j\in\{1,\dots,\bar N\}$, which gives the desired result in the end.
\end{proof}

Now, we deal with the Schauder estimates in the case of the non-divergent fractional
Laplacian, corresponding to~$k:=0$ in Theorem~\ref{SCHAUDER}.
This case is dealt with explicitly in the following result:

\begin{proposition}\label{SCHAUDER=k0}[Sharp Schauder estimates
for the classical fractional Laplacian]
Let $s\in(0,1)$,
$u:\R^n\to\R$ and~$f:B_1\to\R$.

Assume that $u$ is continuous in~$B_1$ and
that
\begin{equation*}
J_{u,0}:=\int_{B_{1/2}^c} \frac{|u(y)|}{|y|^{n+2s}}\,dy 
<+\infty.\end{equation*}
Suppose that
$$ (-\Delta)^s u = f \qquad{\mbox{ in }}\;B_1$$
in the viscosity sense.
Then, for any $\gamma>0$ for which~$\gamma\not\in\N$ and
$\gamma+2s\not\in\N$,
\begin{equation}\label{DER12}
\|u\|_{C^{\gamma+2s}(B_{1/2})}
\le C\, \Big( [f]_{C^\gamma(B_1)}+J_{u,0}\Big),
\end{equation}
for some $C>0$, only depending on $n$, $s$ and~$\gamma$.
\end{proposition}

\begin{proof} Since we are dealing with interior estimates,
up to resizing balls, we will assume that
\begin{equation}\label{w897iugfuiwqegf}
(-\Delta)^s u = f \qquad{\mbox{ in }}\;B_4.\end{equation}
We take~$\tau\in C^\infty_0(B_3,\,[0,1])$ with~$\tau=1$ in~$B_2$
and we set~$v:=\tau u$. We also define
\begin{eqnarray*}&&
\tilde f(x) :=\int_{\R^n}\frac{(1-\tau(y))\,u(y)}{|x-y|^{n+2s}}\,dy=
\int_{B_2^c}\frac{(1-\tau(y))\,u(y)}{|x-y|^{n+2s}}\,dy\\
{\mbox{and }}&&
g(x):=f(x)+\tilde f(x).
\end{eqnarray*}
Hence, from \eqref{w897iugfuiwqegf}, we see that, for any~$x\in B_{3/2}$,
\begin{equation*}
(-\Delta)^s v(x)=
\int_{\R^n} \frac{u(x)-\tau(y)\,u(y)}{|x-y|^{n+2s}}
= g(x).\end{equation*}
Then, from Lemma \ref{SUBP},
we know that
\begin{equation} \label{XDAG:1}
\|v\|_{C^{\gamma+2s}(B_1)} \le C\,\Big(
\|v\|_{L^\infty(\R^n)}+\|g\|_{C^\gamma(B_{3/2})}
\Big),\end{equation}
for some~$C>0$.

Now we observe that, from~\eqref{w897iugfuiwqegf}
and Lemma~\ref{SCk91001},
\begin{equation}  \label{XDAG:2}
\|v\|_{L^\infty(\R^n)}=
\|v\|_{L^\infty(B_3)}\le
\|u\|_{L^\infty(B_3)}
\le C\,\left(
[f]_{C^\gamma(B_4)}+
\int_{B_2^c}\frac{|u(y)|}{|y|^{n+2s}}\,dy
\right),\end{equation}
up to renaming~$C>0$.

Also, for any~$m\in\N$ and any~$x$, $\bar x\in B_{3/2}$,
\begin{eqnarray*}
&&
|D^m\tilde f(x) |\le C
\int_{B_2^c}\frac{|u(y)|}{|x-y|^{n+2s+m}}\,dy\\
{\mbox{and }}&&
|D^m\tilde f(x) -D^m \tilde f(\bar x)|\le C\,|x-\bar x|\,
\int_{B_2^c}\frac{|u(y)|}{|x-y|^{n+2s+m+1}}\,dy
,
\end{eqnarray*}
with~$C>0$ depending on~$m$. As a consequence,
$$ \|\tilde f\|_{C^\gamma(B_{3/2})}\le
C\,\int_{B_2^c}\frac{|u(y)|}{|y|^{n+2s}}\,dy$$
and therefore
\begin{equation}  \label{XDAG:3}
\|g\|_{C^\gamma(B_{3/2})}\le
\|f\|_{C^\gamma(B_{3/2})}+
\|\tilde f\|_{C^\gamma(B_{3/2})}
\le \|f\|_{C^\gamma(B_{3/2})}+
C\,\int_{B_2^c}\frac{|u(y)|}{|y|^{n+2s}}\,dy.
\end{equation}
We also observe that~$u=v$ in~$B_1$ and thus
\begin{equation} \label{XDAG:4}
\|v\|_{C^{\gamma+2s}(B_1)}=\|u\|_{C^{\gamma+2s}(B_1)}.\end{equation}
So, we insert~\eqref{XDAG:2}, \eqref{XDAG:3} and~\eqref{XDAG:4}
into~\eqref{XDAG:1} and we conclude that
\begin{equation}\label{78udhfhdsjfdhhfhh8222}
\|u\|_{C^{\gamma+2s}(B_1)} \le 
C\,\left( \|f\|_{C^\gamma(B_{3/2})}+
[f]_{C^\gamma(B_4)}+
\int_{B_2^c}\frac{|u(y)|}{|y|^{n+2s}}\,dy
\right),\end{equation}
for some~$C>0$.

Also, from~\eqref{w897iugfuiwqegf} and Lemma~\ref{SCk91001},
if we write~$\gamma=m+\beta$, with~$m:=\lfloor\gamma\rfloor$
and~$\beta\in(0,1)$, we have that
$$ \|f\|_{C^m(B_{3/2})}\le C\,\left( [f]_{C^{m+\beta}(B_4)}
+\int_{B_2^c}\frac{|u(y)|}{|y|^{n+2s}}\,dy\right).$$
Therefore, summing~$[f]_{C^{m+\beta}(B_{3/2})}=
[f]_{C^{\gamma}(B_{3/2})}$ to both sides of this inequality,
we find that
$$ \|f\|_{C^\gamma(B_{3/2})}\le C\,\left( [f]_{C^{\gamma}(B_4)}
+\int_{B_2^c}\frac{|u(y)|}{|y|^{n+2s}}\,dy\right).$$
So, we plug this information into~\eqref{78udhfhdsjfdhhfhh8222}
and we conclude that
$$ \|u\|_{C^{\gamma+2s}(B_1)} \le
C\,\left( 
[f]_{C^\gamma(B_4)}+
\int_{B_2^c}\frac{|u(y)|}{|y|^{n+2s}}\,dy
\right),$$
up to renaming~$C>0$, and this is~\eqref{DER12},
after resizing balls.
\end{proof}

{F}rom Proposition~\ref{SCHAUDER=k0}, we obtain
a Schauder estimate for the cutoff equation, as detailed in the following result:

\begin{proposition}\label{PRE:S}
Let $s\in(0,1)$, $k\in\N$ and $u:\R^n\to\R$.

Assume that $u\in C^\alpha_{\rm loc}(B_1)$ for some $\alpha>2s$ and
that
\begin{equation*}
J_{u,k}:=\int_{B_{1/2}^c} \frac{|u(y)|}{|y|^{n+2s+k}}\,dy < +\infty
.\end{equation*}
Let $f_u$ be as in Corollary \ref{POLLI2}.
Let also
\begin{equation}\label{0u9sd13243}
\gamma> k-1\end{equation} such that~$\gamma\not\in\N$ and
$\gamma+2s\not\in\N$. Then, it holds that
\begin{equation}\label{4.15bis}
\|u\|_{C^{\gamma+2s}(B_{1/2})}
\le C\, \Big( [f_u]_{C^\gamma(B_1)}+J_{u,k}\Big),
\end{equation}
for some $C>0$, only depending on $n$, $s$, $\gamma$
and $k$.
\end{proposition}

\begin{proof} We write~$\gamma=m+\theta$, with~$m\in\N$
and~$\theta\in(0,1)$. {F}rom~\eqref{0u9sd13243}, we infer that
\begin{equation}\label{891017s4}
m\ge k-1.
\end{equation}
We take a family of cutoffs $\chi_R$ as in \eqref{9d01} and we exploit
Theorem \ref{POLLI} with $\tau:=\chi_4$. Then, if we set $v:=\chi_4 u$,
we obtain that, for any $x\in B_1$,
\begin{equation}\label{9w8dgwiyfushjvsssg}
(-\Delta)^s v(x)= P_{u,\chi_4}(x)+f_{u,\chi_4}(x),\end{equation}
and $P_{u,\chi_4}$ is a polynomial of degree at most $k-1$.

In particular, from~\eqref{891017s4}, we see that~$[D^mP_{u,\chi_4}]_{C^\theta(B_1)}$
vanishes. Thus, from~\eqref{9w8dgwiyfushjvsssg}
and \eqref{DER12}, we find that
\begin{equation}\label{09auisydghvyqduwfghjdvBISBIS}
\begin{split}
\|u\|_{C^{\gamma+2s}(B_{1/2})}\,&=
\|v\|_{C^{\gamma+2s}(B_{1/2})}\\&
\le C\, \Big( [P_{u,\chi_4}+f_{u,\chi_4}]_{C^\gamma(B_1)}+J_{v,0}\Big)\\
&=C\, \Big( [f_{u,\chi_4}]_{C^\gamma(B_1)}+J_{v,0}\Big).\end{split}\end{equation}
Now we set
$$ \tilde f(x):=
\int_{B_2^c} \frac{(\chi_4(y)-1)\; u(y)\;\psi(x,y)}{|y|^{n+2s+k}}\,dy,$$
where $\psi$ is as in Theorem \ref{POLLI}.
Notice that
$$ D^m\tilde f(x)=
\int_{B_2^c} \frac{(\chi_4(y)-1)\; u(y)\;D^m\psi(x,y)}{|y|^{n+2s+k}}\,dy,$$
and therefore
\begin{equation}\label{09auisydghvyqduwfghjdv}
[\tilde f]_{C^\gamma(B_1)}\le C\,
\int_{B_2^c} \frac{| u(y)|}{|y|^{n+2s+k}}\,dy\le C\,J_{u,k}
\end{equation}
for some $C>0$
(notice that the dependence of $C$ on $\psi$ here is inessential,
due to Remark \ref{RE0}).

Also, from Remark \ref{RE1}, \eqref{f3dif} and \eqref{F3U}, we know that
$$ f_{u,\chi_4}= f_u-f_{3,u}+f^\star_{u,\chi_4}= f_u+\tilde f.$$
This and~\eqref{09auisydghvyqduwfghjdv} imply that
\begin{equation}\label{09auisydghvyqduwfghjdvBIS}
[f_{u,\chi_4}]_{C^\gamma(B_1)}\le [f_u]_{C^\gamma(B_1)}+C\,J_{u,k}.
\end{equation}
Furthermore,
\begin{eqnarray*}
&& J_{v,0}=
\int_{B_{1/2}^c} \frac{| (\chi_4 u)(y)|}{|y|^{n+2s}}\,dy 
\le \int_{B_4\setminus B_{1/2}} \frac{|(\chi_4 u)(y)|}{|y|^{n+2s}}\,dy\\
&&\qquad\qquad\le C\, \int_{B_4\setminus B_{1/2}} 
\frac{|(\chi_4 u)(y)|}{|y|^{n+2s+k}}\,dy\le C\,J_{u,k}.
\end{eqnarray*}
So, we insert this and~\eqref{09auisydghvyqduwfghjdvBIS} into~\eqref{09auisydghvyqduwfghjdvBISBIS}
and we obtain the desired result.
\end{proof}

By combining Definition \ref{DE:k}
and Proposition \ref{PRE:S}, we obtain:

\begin{corollary}\label{gamma barra}
Let $s\in(0,1)$, $k\in\N$, $u:\R^n\to\R$ and~$f:B_1\to\R$.

Assume that $u$ is continuous in~$B_1$ and
\begin{equation*}
J_{u,k}:=\int_{B_{1/2}^c} \frac{|u(y)|}{|y|^{n+2s+k}}\,dy < +\infty
.\end{equation*}

Suppose that
\begin{equation}\label{o90:x0}
(-\Delta)^s u \ugu f \qquad{\mbox{ in }}\;B_1.\end{equation}
Then, for any 
\begin{equation}\label{c ga91}
\gamma>k-1\end{equation}
such that~$\gamma\not\in\N$ and
$\gamma+2s\not\in\N$, it holds that
\begin{equation}\label{GH:OAA BIS}
\|u\|_{C^{\gamma+2s}(B_{1/2})}
\le C\, \Big( [f]_{C^\gamma(B_1)}+J_{u,k}\Big),
\end{equation}
for some $C>0$, only depending on $n$, $s$, $\gamma$
and $k$.
\end{corollary}

\begin{proof}
First of all, we prove the result under the additional assumption that
\begin{equation}\label{89sduiv1111}
{\mbox{$u\in C^\alpha_{\rm loc}(B_1)$ for some $\alpha>2s$.}}
\end{equation}
In this case, we fall under the assumptions of
Remark \ref{RE2}, and so we have that 
\begin{equation}\label{o90:x1}
(-\Delta)^s u\ugu f_u\quad{\mbox{ in }}B_1.\end{equation} 
Also,
\begin{equation}\label{o90:x0}
(-\Delta)^s u\ugu f\quad{\mbox{ in }}B_1,\end{equation}
Consequently, by \eqref{o90:x0} and~\eqref{o90:x1}, in view of
Lemma~\ref{UGUP},
it follows that $f-f_u$ is a polynomial $P_u$ of degree at most $k-1$.

{F}rom this and \eqref{c ga91}, we obtain that
$[f_u]_{C^\gamma(B_1)}=[f]_{C^\gamma(B_1)}$.
This and \eqref{4.15bis}
imply 
\eqref{GH:OAA BIS}.

Now we consider the general case. For this, we take~$\rho\in C^\infty_0(B_1)$
and consider the mollifier $\rho_\eps(x):=\eps^{-n}\rho(x/\eps)$.
We consider the convolutions~$u_\eps:= u*\rho_\eps$ and~$f_\eps:=f*\rho_\eps$
and we know (see e.g. formula~(3.2)
in~\cite{MR3161511}) that~$(-\Delta)^s u_\eps=f_\eps$ in~$B_{99/100}$,
as long as~$\eps$ is small enough.
Since~\eqref{89sduiv1111} is satisfied by~$u_\eps$, we can apply
the result already established and conclude that, up to resizing balls,
\begin{equation}\label{HJ01:X01}
\|u_\eps\|_{C^{\gamma+2s}(B_{9/10})}
\le C\, \left( [f_\eps]_{C^\gamma(B_{9/10})}+
\int_{B_{3/4}^c}\frac{|u_\e(y)|}{
|y|^{n+2s+k}}\,dy\right),
\end{equation}
for some~$C>0$. In particular,
$u_\e$ converges to~$u$ in~$C^{\gamma+2s}(B_{1/2})$
and, by taking limits, we have that
\begin{equation}\label{HJ01:X02}
\lim_{\e\to0} \|u_\eps\|_{C^{\gamma+2s}(B_{9/10})}\ge
\|u\|_{C^{\gamma+2s}(B_{1/2})} \quad{\mbox{ and }}\quad
\lim_{\e\to0} [f_\eps]_{C^\gamma(B_{9/10})}\le
[f]_{C^\gamma(B_{1})}.
\end{equation}
Furthermore, if~$y\in {B_{3/4}^c} $ and~$\xi\in {B_\e(y)}$, we have that
$$ |\xi|\le |y|+|\xi-y|\le |y|+\e\le 2|y|,$$
and therefore
\begin{equation}\label{HJ01:X03}
\begin{split}
&\int_{B_{3/4}^c}\frac{|u_\e(y)|}{
|y|^{n+2s+k}}\,dy \le
\int_{B_{3/4}^c} \left[ \int_{B_\e(y)}\frac{|u(\xi)|\,|\rho_\e(y-\xi)|}{
|y|^{n+2s+k}}\,d\xi \right]\,dy
\\ &\quad\qquad\le C\,
\int_{B_{3/4}^c} \left[ \int_{B_\e(y)}\frac{|u(\xi)|\,|\rho_\e(y-\xi)|}{
|\xi|^{n+2s+k}}\,d\xi \right]\,dy\le C\,
\int_{B_{1/2}^c} \left[ \int_{\R^n}\frac{|u(\xi)|\,|\rho_\e(y-\xi)|}{
|\xi|^{n+2s+k}}\,dy \right]\,d\xi\\ &\quad\qquad
=C\,
\int_{B_{1/2}^c} \frac{|u(\xi)|}{
|\xi|^{n+2s+k}}\,d\xi = CJ_{u,k}.
\end{split}
\end{equation}
So we plug \eqref{HJ01:X02} and~\eqref{HJ01:X03} into~\eqref{HJ01:X01}
and we obtain~\eqref{GH:OAA BIS}.
\end{proof}

With this we are now in the position of giving the proof
of Theorem \ref{SCHAUDER}:

\begin{proof}[Proof of Theorem \ref{SCHAUDER}] 
We claim that
\begin{equation}\label{788:103:01cvbnnnnn}
\|u\|_{C^{\gamma+2s}(B_{1/2})} \le
C\,\Big([f]_{C^\gamma(B_1;k)}+J_{u,k}\Big),
\end{equation}
for some~$C>0$.
We observe that when $\gamma>k-1$ the claim in~\eqref{788:103:01cvbnnnnn}
follows from Corollary~\ref{gamma barra}
and~\eqref{TOGLI}. Hence, we can now focus on the case in which
\begin{equation}\label{minore}
\gamma < k-1.\end{equation}
We take $v$ to be a solution of
\begin{equation}\label{1sdf-pokhjodv3}
(-\Delta)^s v=f\quad{\mbox{ in }}B_1,\end{equation}
with $v=0$ in $B_1^c$.

Then, from Proposition 1.1 in \cite{MR3168912}, we have that
\begin{equation}\label{788:103:01}
\| v\|_{C^s(\R^n)}\le C\,\|f\|_{L^\infty(B_1)},
\end{equation}
for some $C>0$. 

Also, from Proposition \ref{SCHAUDER=k0},
\begin{equation}\label{788:103:02}
\| v\|_{C^{\gamma+2s}(B_{1/2})}\le C\,\Big([f]_{C^\gamma(B_1)}+J_{v,0}\Big).
\end{equation}
Since, from \eqref{788:103:01},
\begin{equation}\label{788:103:02222} J_{v,0}\le C\,
\| v\|_{L^\infty(\R^n)}\le C\,\|f\|_{L^\infty(B_1)},\end{equation}
we deduce from \eqref{788:103:02} that
\begin{equation}\label{788:103:03}
\| v\|_{C^{\gamma+2s}(B_{1/2})}\le C\,\Big([f]_{C^\gamma(B_1)}+
\|f\|_{L^\infty(B_1)}\Big).
\end{equation}
Also, from \eqref{1sdf-pokhjodv3} and Corollary \ref{8IJAJS2wedfv1234},
we have that $(-\Delta )^s v
\;{\stackrel{0}{=}}\;f$ in $B_1$.

{F}rom this and \eqref{9idwfjdv827rutgjsakhdaaaaa2},
we conclude that
$(-\Delta)^s v
\;{\stackrel{k}{=}} \;f$ in $B_1$.

So, we define $w:=u-v$ and we have that $(-\Delta)^s w
\;{\stackrel{k}{=}} \;0$ in $B_1$. Hence, we take $\bar\gamma:=k-1+\e$, for a fixed,
small $\e>0$, and we are in the position of using
Corollary \ref{gamma barra} (notice indeed that $\bar\gamma$ satisfies
\eqref{c ga91}). In this way, we obtain that
\begin{equation} \label{8eywdhiujchgsgdggdd}
\|w\|_{C^{\bar\gamma+2s}(B_{1/2})}\le C\, J_{w,k}.\end{equation}
We also point out that
\begin{equation}\label{8eywdhiujchgsgdggdd2}
J_{w,k}\le J_{u,k}+J_{v,k}\le
J_{u,k}+C\,\|f\|_{L^\infty(B_1)},\end{equation}
where \eqref{788:103:02222} has been used once again.

Also, $\bar\gamma+2s\ge \gamma+2s$, due to~\eqref{minore}, and so
$$ \|w\|_{C^{\bar\gamma+2s}(B_{1/2})} \ge \|w\|_{C^{\gamma+2s}(B_{1/2})}
\ge \|u\|_{C^{\gamma+2s}(B_{1/2})} - \|v\|_{C^{\gamma+2s}(B_{1/2})}.$$
Using this, \eqref{8eywdhiujchgsgdggdd}
and \eqref{8eywdhiujchgsgdggdd2}, we find
$$ \|u\|_{C^{\gamma+2s}(B_{1/2})} \le
C\,\Big(\|v\|_{C^{\gamma+2s}(B_{1/2})}+ 
J_{u,k}+\|f\|_{L^\infty(B_1)}\Big).$$
This and \eqref{788:103:03} imply that
$$ \|u\|_{C^{\gamma+2s}(B_{1/2})} \le
C\,\Big([f]_{C^\gamma(B_1)}+
\|f\|_{L^\infty(B_1)}+J_{u,k}\Big).$$
Now, since this estimate is valid for $f$ satisfying $(-\Delta)^s u\ugu f$,
it must be valid also for $f+P$, for any polynomial $P$ of degree $k-1$
(recall \eqref{0dfhvnoooP}).
Consequently, we can write
$$ \|u\|_{C^{\gamma+2s}(B_{1/2})} \le
C\,\inf \Big([f+P]_{C^\gamma(B_1)}+
\|f+P\|_{L^\infty(B_1)}+J_{u,k}\Big).$$
{F}rom this and Lemma \ref{Xuax182rrrr3}, it follows that~\eqref{788:103:01cvbnnnnn}
holds true, as desired.

We remark that~\eqref{788:103:01cvbnnnnn}
is indeed the desired result in \eqref{GH:OAA}, except that we wish to replace 
$[f]_{C^\gamma(B_1;k)}$ with $[f]_{C^\gamma(B_1;\ell)}$ and
$J_{u,k}$
with $J_{u,\ell}$.

For this, we observe that both $[f]_{C^\gamma(B_1;j)}$
and $J_{u,j}$ are decreasing in $j\in\N$ (up to multiplicative constants).
Hence, when $\ell\le k$, then \eqref{GH:OAA} follows directly from~\eqref{788:103:01cvbnnnnn}.

On the other hand, when $\ell>k$ we
see that $(-\Delta)^s u
\;{\stackrel{\ell}{=}}\;f$ in $B_1$, thanks to \eqref{9idwfjdv827rutgjsakhdaaaaa2}.
So we can apply \eqref{788:103:01cvbnnnnn} with $\ell$ replacing $k$,
which is the desired result in \eqref{GH:OAA}.
\end{proof}

\section{Liouville Theorem
for divergent fractional Laplacians}\label{SECT:4}

By using the Schauder estimates in Theorem \ref{SCHAUDER} at any scale,
we can now give the proof
of Theorem \ref{LIOUVILLE}.

\begin{proof}[Proof of Theorem \ref{LIOUVILLE}]
We first suppose that $(-\Delta)^s u\ugu 0$ in $\R^n$
and we show that $u$ is necessarily a polynomial of degree at most $d(k,s)$.
For this, we take
$$ \gamma:=\left\{\begin{matrix} k+2-2s & {\mbox{ if }} s\in\left( \frac12,\,1\right),\\
k+1-2s & {\mbox{ if }} s\in\left( 0,\frac12\right].\end{matrix}
\right.$$
Notice that~$\gamma+2s> k$ and
$$ m:=
\left \lfloor{\gamma+2s}\right \rfloor 
=
\left\{\begin{matrix} k+2& {\mbox{ if }} s\in\left( \frac12,\,1\right),\\
k+1& {\mbox{ if }} s\in\left( 0,\frac12\right].\end{matrix}
\right.$$
In particular, we have that 
\begin{equation}\label{m grande}
m\ge k+2s.\end{equation}
Now, for any $j\in\N$, $j\ge1$, we define $u_j(x):=u(jx)$.
Then, $(-\Delta)^s u_j\ugu 0$ in $B_1$, hence
Theorem \ref{SCHAUDER} gives that
\begin{eqnarray*}
&& \|D^mu\|_{L^\infty(B_{j/2})}
=j^{-m} \|D^m u_j\|_{L^\infty(B_{1/2})}
\le
j^{-m} \|u_j\|_{C^{\gamma+2s}(B_{1/2})}
\le C\,j^{-m}\, J_{u_j,k}\\
&&\qquad\quad=
C\,j^{2s+k-m} \int_{B_{j/2}^c} \frac{|u(y)|}{|y|^{n+2s+k}}\,dy.
\end{eqnarray*}
So we can send $j\to+\infty$ and use \eqref{m grande} to see that $D^mu$ vanishes
identically, hence $u$ is a polynomial
of degree less than or equal to $m-1$, as desired.

Now, we prove the converse statement.
Namely, we show that 
\begin{equation}\label{CONV-LIOUVILLE}\begin{split}
&{\mbox{all the polynomials $P$ of degree at most $d(k,s)$}}
\\&{\mbox{satisfy $(-\Delta)^s P\ugu 0$ in $\R^n$.}}\end{split}\end{equation}
The proof of this is by induction over $k$. If $k=0$, then $d(k,s)=1$
if $s\in\left( \frac12,\,1\right)$ and $d(k,s)=0$
if $s\in\left( 0,\frac12\right]$. Hence, if $P$ has degree at most $d(k,s)$,
it follows that $P$ is affine if $s\in\left( \frac12,\,1\right)$
and constant if $s\in\left( 0,\frac12\right]$, and
$$ \int_{\R^n} \frac{|P(y)|}{1+|y|^{n+2s}}\,dy <+\infty.$$
In any case, $(-\Delta)^s P$ is well defined in the standard sense, and
$(-\Delta)^s P=0$ in $\R^n$. Accordingly, by 
Corollary \ref{8IJAJS2wedfv1234}, we have that
$(-\Delta )^s P\;{\stackrel{0}{=}}\;0$ in $\R^n$.

This is the desired result when $k=0$. Hence, we now suppose
recursively that the claim in \eqref{CONV-LIOUVILLE} holds true for $k-1$
and we prove it for $k$.

For this, we take a polynomial $P$ with degree at most $d(k,s)$
and, for any fixed $i\in\{1,\dots,n\}$, we set $Q_i:=\partial_i P$.
Notice that $Q_i$ is a polynomial with degree at most $d(k,s)-1=d(k-1,s)$.
Therefore, by the inductive hypothesis we know that
\begin{equation}\label{89iucve7e7e7hhfhfa}
(-\Delta )^s Q_i\;{\stackrel{k-1}{=}}\;0 \quad{\mbox{ in }} \R^n.
\end{equation}
Furthermore, by Theorem \ref{POLLI}
and Remark \ref{RE1} (recall also Remark~\ref{RE3}), we know that, fixed~$M>0$, for any large $R>0$,
\begin{equation} \label{3oghtgf76w56w7tfg84ytiufi}
(-\Delta)^s(\tau_R P)= f_P+ g_{R} + P_R\qquad{\mbox{in }}B_M,\end{equation}
where $\tau_R\in C^\infty_0(B_R,\,[0,1])$ with~$\tau_R=1$ in~$B_{R-1}$
and~$\|\nabla\tau_R\|_{L^\infty(\R^n)} \le 4$,
$P_R$ is a polynomial of degree $k-1$ and
$$ g_{R}(x):=
\int_{B_{R-1}^c} \frac{(\tau_R(y)-1)\,P(y)\,\psi(x,y)}{|y|^{n+2s+k}}\,dy.$$
We define
$$ \zeta_{R,i}:= (-\Delta)^s(\partial_i\tau_R P).$$
We claim that
\begin{equation}\label{pez}\begin{split}&
{\mbox{$\zeta_{R,i}=\tilde P_{R,i}+\tilde\zeta_{R,i}\;$, where~$\tilde P_{R,i}$
is a polynomial of degree $k-2$}}\\&{\mbox{and $\tilde\zeta_{R,i}
\to0\;\;$ in $\;B_M\;\;$ as $\;R\to+\infty$.}}\end{split}
\end{equation}
To check this, we observe that
$$ \partial_i\tau_R P = \tau_{R+1}\partial_i\tau_R P.$$
Thus, fixed~$M$, we can use Theorem~\ref{POLLI} (with~$\tau:=\tau_{R+1}$,
$u:=\partial_i\tau_R P$ and~$k$ replaced by~$k-1$)
and find that, for any~$x\in B_M$,
\begin{eqnarray*}
(-\Delta)^s (\partial_i\tau_R P)(x)&=& (-\Delta)^s (\tau_{R+1}\partial_i\tau_R P)(x)\\
&= &\tilde P_{R,i}(x)+\int_{B_{2M}} \frac{(\partial_i\tau_R P)(x)-
(\partial_i\tau_R P)(y)}{|x-y|^{n+2s}}\,dy+\int_{B_{2M}^c} 
\frac{(\partial_i\tau_R P)(x)}{|x-y|^{n+2s}}\,dy\\ &&\qquad+
\int_{B_{2M}^c} \frac{\tau_{R+1}(y)\, (\partial_i\tau_R P)(y)\;\psi(x,y)}{|y|^{n+2s+k-1}}
\,dy,
\end{eqnarray*}
for some polynomial $\tilde P_{R,i}$,
which has degree at most $k-2$.
Now, for large~$R$, the terms supported in~$B_{2M}$ vanish, namely
we can write that
$$ (-\Delta)^s (\partial_i\tau_R P)(x)=
\tilde P_{R,i}(x)+
\int_{B_R\setminus B_{R-1}} \frac{\tau_{R+1}(y)\, 
(\partial_i\tau_R P)(y)\;\psi(x,y)}{|y|^{n+2s+k-1}}
\,dy=\tilde P_{R,i}(x)+\tilde\zeta_{R,i}(x),$$
with
$$ \tilde\zeta_{R,i}(x):=\int_{B_R\setminus B_{R-1}} \frac{\tau_{R+1}(y)\, 
(\partial_i\tau_R P)(y)\;\psi(x,y)}{|y|^{n+2s+k-1}}
\,dy.$$
Hence, to prove~\eqref{pez}, we need to show that
\begin{equation}\label{hg345dfjvcadgojasg}
{\mbox{$\tilde\zeta_{R,i}
\to0\;\;$ in $\;B_M\;\;$ as $\;R\to+\infty$.}}
\end{equation}
To this aim, we recall \eqref{PSI BUONA} and we
compute, for large~$R$,
\begin{eqnarray*}
&& \left| 
\int_{B_R\setminus B_{R-1}} \frac{\tau_{R+1}(y)\, 
(\partial_i\tau_R P)(y)\;\psi(x,y)}{|y|^{n+2s+k-1}}
\,dy \right|
\le C\,
\int_{B_R\setminus B_{R-1}} \frac{|P(y)|}{|y|^{n+2s+k-1}}
\,dy \\ &&\qquad\qquad\le
C\,
\int_{B_R\setminus B_{R-1}} \frac{R^{d(k,s)}}{|y|^{n+2s+k-1}}
\,dy \le C\,R^{d(k,s) -2s-k},
\end{eqnarray*}
up to renaming~$C$ at any step. The latter quantity is infinitesimal
as~$R\to+\infty$, thanks to~\eqref{dksdef}.
This establishes~\eqref{hg345dfjvcadgojasg}, and so~\eqref{pez}.

Notice also that
$$ \partial_i (-\Delta)^s(\tau_R P) =
(-\Delta)^s(\partial_i\tau_R P)+(-\Delta)^s(\tau_R \partial_i P)
=(-\Delta)^s(\partial_i\tau_R P)+(-\Delta)^s(\tau_R Q_i).$$
Accordingly, 
by~\eqref{3oghtgf76w56w7tfg84ytiufi} and~\eqref{pez}, we obtain that,
in $B_{M}$, 
\begin{equation} \label{iojdlv2345fb}
(-\Delta)^s(\tau_R Q_i)=\partial_i (-\Delta)^s(\tau_R P)-(-\Delta)^s(\partial_i\tau_R P)
= \partial_i f_P+ \partial_i g_{R} + \partial_i P_R-\tilde P_{R,i}-\tilde\zeta_{R,i}.\end{equation}
Notice that, in view of \eqref{PSI BUONA}, we have that $\partial_i g_{R}\to0$.
Also, $\partial_i P_R$ is a polynomial of degree $k-2$.
In consequence of these observations and \eqref{iojdlv2345fb}, 
we have that~$(-\Delta )^s Q_i\;{\stackrel{k-1}{=}}\; \partial_i f_P$
in~$B_M$. 
{F}rom this, \eqref{89iucve7e7e7hhfhfa} and
Lemma \ref{UGUP}, we obtain that
there exists a polynomial $Q^\star_i$ of degree at most $k-2$ such that $
\partial_i f_P=Q^\star_i$.

This implies that, in~$B_M$,
\begin{equation}\label{osjcJJ1234JA}
{\mbox{$f_P$ is a polynomial of degree at most $k-1$.}}\end{equation}
On the other hand, from Remark \ref{RE2},
we know that $(-\Delta)^s P\ugu f_P$ in $B_M$.
Using this and \eqref{osjcJJ1234JA}, and recalling \eqref{0dfhvnoooP},
we can write $(-\Delta)^s P\ugu 0$ in $B_M$.
Since~$M>0$ is arbitrary, it follows that~$(-\Delta)^s P\ugu 0$ in~$\R^n$,
as desired.
\end{proof}

\section{Stability of divergent fractional Laplacians}\label{STABS}

The goal of this section is to prove Theorem \ref{STAB}, namely
that the divergent fractional Laplacian is stable under limits
that are compatible with the viscosity setting.
For this, 
we first consider the simpler case in which the functions
vanish in~$B_1$ (the advantage of this setting being that
the smoothness assumption in Remark~\ref{RE2} is obviously
satisfied). The precise result goes as follows:

\begin{lemma}\label{vkpre}
Let $s\in(0,1)$ and $k\in\N$.
Let us consider sequences of functions $v_m:\R^n\to\R$ and $g_m:B_1\to\R$ such that
$v_m=0$ in $B_1$
and $g_m$ is continuous in~$B_1$, with
\begin{equation}\label{ususKKK78uidaaa3r}\sup_{m\in\N}
\int_{\R^n} \frac{|v_m(y)|}{1+|y|^{n+2s+k}}\,dy < +\infty
,\end{equation}
and
\begin{equation}\label{89dw9duuuuu78uidaaa3r9quwehfg}
(-\Delta)^s v_m \ugu g_m \qquad{\mbox{ in }}\;B_1.\end{equation}
Assume that $g_m\to g$ a.e. in $B_1$
as $m\to +\infty$, for some function~$g:B_1\to\R$.

Suppose also that 
\begin{equation}\label{9isckneyeyeyyehsh}
\lim_{m\to+\infty} \int_{B_1^c} \frac{v_m(y)\,\varphi(y)}{|y|^{n+2s+k}}\,dy=
\int_{B_1^c} \frac{v(y)\,\varphi(y)}{|y|^{n+2s+k}}\,dy\end{equation}
for any $\varphi\in L^\infty(B_1^c)$,
for some function $v:\R^n\to\R$ with~$v=0$ in~$B_1$.

Then, it holds that
\begin{equation}\label{9iuwde83ueghhfhfhfhf536475890yogjfhh11}
(-\Delta)^s v\ugu g \qquad{\mbox{ in }}\;B_1.\end{equation}
\end{lemma}

\begin{proof} We can use Remark~\ref{RE2} and~\eqref{FU1} and find that, for any $x\in B_1$,
$$ (-\Delta)^s v_m(x)\ugu f_{v_m}(x)
=-\int_{B_2\setminus B_1} \frac{v_m(y)}{|x-y|^{n+2s}}\,dy+
\int_{B_2^c} \frac{v_m(y)\,\psi(x,y)}{|y|^{n+2s+k}}\,dy.$$
{F}rom this, \eqref{89dw9duuuuu78uidaaa3r9quwehfg} and
Lemma \ref{UGUP}, we obtain that
$$ g_{m}(x)=-\int_{B_2\setminus B_1} \frac{v_m(y)}{|x-y|^{n+2s}}\,dy+
\int_{B_2^c} \frac{v_m(y)\,\psi(x,y)}{|y|^{n+2s+k}}\,dy +P_m(x),$$
where $P_m$ is a polynomial of degree at most~$k-1$.

We stress that, fixed~$x\in B_1$, 
$$ \inf_{y\in B_2\setminus B_1}|x-y|\ge \inf_{y\in B_2\setminus B_1}|y|-|x|=1-|x|,$$
and so the function~$y\mapsto
\frac{1}{|x-y|^{n+2s}}$ belongs to~$L^\infty(B_2\setminus B_1)$.
Thus, in view of~\eqref{PSI BUONA}
and~\eqref{9isckneyeyeyyehsh}, we have that, for any fixed~$x\in B_1$,
$$ \lim_{m\to+\infty} P_m(x)=g(x)+\int_{B_2\setminus B_1} \frac{v(y)}{|x-y|^{n+2s}}\,dy-
\int_{B_2^c} \frac{v(y)\,\psi(x,y)}{|y|^{n+2s+k}}\,dy.$$
This and Lemma~\ref{9uids2345hv} imply that
there exists a polynomial~$P$ of degree at most~$k-1$ such that
\begin{equation}\label{9s9s14256378tifjdhgsfpqpqpq}
P(x)=g(x)+
\int_{B_2\setminus B_1} \frac{v(y)}{|x-y|^{n+2s}}\,dy
-\int_{B_2^c} \frac{v(y)\,\psi(x,y)}{|y|^{n+2s+k}}\,dy.\end{equation}
Also, using~\eqref{9isckneyeyeyyehsh} with~$\varphi:=\chi_{(0,+\infty)}(v (y))$,
we see that
$$ \int_{B_1^c} \frac{v_+(y)}{|y|^{n+2s+k}}\,dy
=\lim_{m\to+\infty} \int_{B_1^c} \frac{v_m(y)\,\chi_{(0,+\infty)}(v (y))}{
|y|^{n+2s+k}}\,dy\le \sup_{m\in\N}\int_{B_1^c} \frac{|v_m(y)|}{
|y|^{n+2s+k}}\,dy ,$$
which is finite,
thanks to~\eqref{ususKKK78uidaaa3r}.
With a similar computation on~$v_-$, we thus conclude that
$$ \int_{B_1^c} \frac{|v(y)|}{|y|^{n+2s+k}}\,dy
<+\infty.$$
So, we can use
Remark~\ref{RE2} on~$v$ and obtain
$$ (-\Delta)^s v(x)\ugu f_{v}(x)
=
-\int_{B_2\setminus B_1} \frac{v(y)}{|x-y|^{n+2s}}\,dy+
\int_{B_2^c} \frac{v(y)\,\psi(x,y)}{|y|^{n+2s+k}}\,dy.$$
{F}rom this, \eqref{0dfhvnoooP}
and~\eqref{9s9s14256378tifjdhgsfpqpqpq}, we deduce that~$(-\Delta)^s v(x)\ugu g$,
as desired.
\end{proof}

With this preliminary result, we can complete the proof of
the stability theorem, by arguing as follows:

\begin{proof}[Proof of Theorem \ref{STAB}] We set
\begin{eqnarray*}
&& v_m:=(1-\chi_1)\,u_m, \qquad w_m:=\chi_1 u_m,\\
&& v:=(1-\chi_1)\,u\qquad{\mbox{and}}\qquad w:=\chi_1 u.
\end{eqnarray*}
By construction, $v_m\to v$ and~$w_m\to w$ locally uniformly in~$B_1$, as~$m\to+\infty$.

In light of \eqref{78uidaaa3r9quwehfg} and
Corollary~\ref{2345923847hffhhfhfhfsksk} (used here with~$\rho:=1$), we know that
\begin{equation}\label{Si1230139}
(-\Delta)^s w_m = \bar f_m+\int_{B_2\setminus B_1} \frac{u_m(y)}{|x-y|^{n+2s}}\,dy-
\int_{B_2^c}\frac{u_m(y)\,\psi(x,y)}{|y|^{n+2s+k}}\,dy=:h_m\end{equation}
in~$B_1$ in the sense of viscosity,
where
\begin{equation}\label{8du6105984}
\bar f_m:=f_m+P_m\end{equation} and $P_m$ is a polynomial
of degree at most~$k-1$. Thus, from Corollary~\ref{8IJAJS2wedfv1234} and
\eqref{9idwfjdv827rutgjsakhdaaaaa2}
we obtain that
$$ (-\Delta)^s w_m \ugu \bar f_m+
\int_{B_2\setminus B_1} \frac{u_m(y)}{|x-y|^{n+2s}}\,dy
-\int_{B_2^c}\frac{u_m(y)\,\psi(x,y)}{
|y|^{n+2s+k}}\,dy.$$
Hence, in view of~\eqref{0dfhvnoooP}, we obtain that
$$ (-\Delta)^s w_m \ugu f_m+
\int_{B_2\setminus B_1} \frac{u_m(y)}{|x-y|^{n+2s}}\,dy
-\int_{B_2^c}
\frac{u_m(y)\,\psi(x,y)}{|y|^{n+2s+k}}\,dy.$$
As a consequence, for any~$x\in B_1$,
\begin{equation}\label{9fusususu010187654338475}
\begin{split}& (-\Delta)^s v_m(x)= 
(-\Delta)^s u_m(x)-(-\Delta)^s w_m(x)\\&\qquad\qquad
\ugu f_m(x)-\left(
f_m(x)+
\int_{B_2\setminus B_1} \frac{u_m(y)}{|x-y|^{n+2s}}\,dy-\int_{B_2^c}\frac{u_m(y)\,\psi(x,y)}{|y|^{n+2s+k}}\,dy\right)\\
&\qquad\qquad=-\int_{B_2\setminus B_1} \frac{u_m(y)}{|x-y|^{n+2s}}\,dy+
\int_{B_2^c}\frac{u_m(y)\,\psi(x,y)}{|y|^{n+2s+k}} \,dy=:g_m(x)
.\end{split}\end{equation}
Let also
$$ g(x):=-\int_{B_2\setminus B_1} \frac{u(y)}{|x-y|^{n+2s}}\,dy+
\int_{B_2^c}\frac{u(y)\,\psi(x,y)}{|y|^{n+2s+k}}\,dy.$$
Notice that, by~\eqref{PSI BUONA} and~\eqref{iodsjhdhsuususu4857489},
we have that~$g_m\to g$ pointwise in~$B_1$.

Also, fixed any~$\rho\in(0,1)$, by~\eqref{PSI BUONA},
\begin{eqnarray*}
\sup_{x\in B_\rho}|\nabla g_m(x)|&\le& \sup_{x\in B_\rho} C\,\left(
\int_{B_2\setminus B_1} \frac{|u_m(y)|}{|x-y|^{n+2s+1}}\,dy
+
\int_{B_2^c}\frac{|u_m(y)|\,|\nabla\psi(x,y)|}{|y|^{n+2s+k}}\,dy\right)\\
&\le& C\,\left(
\int_{B_2\setminus B_1} \frac{|u_m(y)|}{(1-\rho)^{n+2s+1}}\,dy+
\int_{B_2^c}\frac{|u_m(y)|}{|y|^{n+2s+k}}\,dy\right),
\end{eqnarray*}
which is bounded uniformly in~$m$,
thanks to~\eqref{iodsjhdhsuususu4857489:PRE}.
Accordingly, by the Theorem of Ascoli,
\begin{equation}\label{89iodc73847812}
{\mbox{$g_m\to g$ locally uniformly in~$B_1$.}}\end{equation}
Thus, from \eqref{9fusususu010187654338475} and Lemma~\ref{vkpre},
we conclude that
\begin{equation}\label{9fusususu010187654338475:2}
(-\Delta)^s v \ugu g\qquad{\mbox{in }}B_1.\end{equation}
Now we prove that
\begin{equation}\label{9fusususu010187654338475:3}
(-\Delta)^s w \ugu f-g \qquad{\mbox{in }}B_1.\end{equation}
For this, we take~$\varphi\in C^\infty_0(B_1)$. We let~$U\Subset B_1$
be the support of~$\varphi$ and we fix~$\e>0$ suitably small
(also in dependence of~$U$ and~$B_1$).
We take~$\rho\in C^\infty_0(B_1)$ and~$\rho_\eps(x):=
\eps^{-n}\rho(x/\eps)$.
We consider the convolutions~$w_{m,\eps}:= w_m
*\rho_\eps$ and~$h_{m,\eps}:=h_m*\rho_\eps$.
Notice that~$w_{m,\eps}$ is smooth and compactly supported in~$B_{11/10}$.
Then (see e.g. formula~(3.2)
in~\cite{MR3161511}) 
we have that~$(-\Delta)^s w_{m,\eps}=h_{m,\eps}$ in~$U$ in the smooth sense.
Therefore we can write that
\begin{eqnarray*}
&& \int_{U}  h_{m,\eps} \varphi
=\int_{U}  (-\Delta)^s w_{m,\eps} \varphi=
\int_{\R^n}  (-\Delta)^s w_{m,\eps} \varphi\\&&\qquad\qquad
=\int_{\R^n}  w_{m,\eps} \,(-\Delta)^s\varphi=
\int_{B_{11/10}}  w_{m,\eps} \,(-\Delta)^s\varphi
.\end{eqnarray*}
Hence, for any~$m'$, $m\in\N$,
$$ \left| \int_{U}  (h_{m,\eps} -h_{m',\eps})\,\varphi\right|
\le
\int_{B_{11/10}}  |w_{m,\eps} -w_{m',\eps}|\,|(-\Delta)^s\varphi|.$$
Since $w_m\in L^1(\R^n)$,
by sending~$\eps\to0$, we thus obtain that
\begin{equation} \label{80234857:01}\begin{split}&
\left| \int_{U}  (h_{m} -h_{m'})\,\varphi\right|
\le
\int_{B_{11/10}}  |w_{m} -w_{m'}|\,|(-\Delta)^s\varphi|\\ &\qquad=
\int_{B_{1}}  |w_{m} -w_{m'}|\,|(-\Delta)^s\varphi|
\le C\,\|w_{m} -w_{m'}\|_{L^1(B_1)} \,\|\varphi\|_{C^2(\R^n)}.\end{split}\end{equation}
{F}rom the convergence of~$u_m$ in~$L^1(B_1)$, we also have that
\begin{equation}\label{80234857:02}
\lim_{m\to+\infty} \|w_{m} -\chi_1 u\|_{L^1(B_1)}=\lim_{m\to+\infty} \|u_{m} - u\|_{L^1(B_1)}
=0.
\end{equation}
{F}rom \eqref{80234857:01} and \eqref{80234857:02}, it follows
that~$h_m$ is a Cauchy sequence in the norm~$ 
\|\cdot\|_\star$ introduced in~\eqref{no0195678tif}.
{F}rom the uniform convergence, we also know that~$f_m$
is a Cauchy sequence in the norm~$ 
\|\cdot\|_\star$. Moreover, by~\eqref{89iodc73847812},
we have that~$g_m$ is also a Cauchy sequence in
the norm~$\|\cdot\|_\star$.

Since
\begin{equation}\label{09w8wiejfhg7ydwghcjweiugf}
P_m=g_m+h_m-f_m,\end{equation} these observations
imply that~$P_m$ is also
a Cauchy sequence in
the norm~$\|\cdot\|_\star$ and so, in consequence of
Lemma~\ref{45678959uids2345hv},
we obtain that~$P_m$ converges uniformly 
to some polynomial $P$ of degree at most $k-1$ in~$U$, for any~$U\Subset B_1$.

This and~\eqref{8du6105984} imply that~$\bar f_m$ converges
locally uniformly in~$B_1$.
Hence, writing~$h_m=\bar f_m-g_m$, we conclude that~$h_m$
also converges
locally uniformly in~$B_1$ to some function $h$.

We are therefore in the position to use Lemma~5 in~\cite{MR2781586}
and deduce from~\eqref{Si1230139} that~$(-\Delta)^s w=h$
in~$B_1$ in the sense of viscosity. Hence, by Corollary~\ref{8IJAJS2wedfv1234},
we can write~$(-\Delta )^s w\;{\stackrel{0}{=}}\;h$ in~$B_1$.

Passing to the limit in \eqref{09w8wiejfhg7ydwghcjweiugf},
we obtain that
$$ P=g+h-f$$
and so~$(-\Delta )^s w\;{\stackrel{0}{=}}\;f-g+P$ in~$B_1$.

{F}rom this,~\eqref{9idwfjdv827rutgjsakhdaaaaa2} and \eqref{0dfhvnoooP},
we conclude that~\eqref{9fusususu010187654338475:3} holds true, as desired.

Now, by~\eqref{9fusususu010187654338475:2} and~\eqref{9fusususu010187654338475:3},
we obtain that
$$ (-\Delta)^su=(-\Delta)^sv+(-\Delta)^sw\ugu g+(f-g)=f,$$
as desired.
\end{proof}

A useful consequence of Theorem \ref{STAB} is also
a stability result under convolution, which goes as follows:

\begin{proposition}
Let $k\in\N$, $s\in(0,1)$. Assume that $u$ and
$f$ are continuous functions in $B_1$, with
\begin{equation} \label{09sdugfnklsfe245}
\int_{\R^n}\frac{|u(y)|}{1+|y|^{n+2s+k}}<+\infty\end{equation}
and
$$ (-\Delta)^s u \ugu f \qquad{\mbox{ in }}\;B_1.$$
Let~$\e>0$, $\rho\in C^\infty_0(B_1)$
and $\rho_\eps(x):=\eps^{-n}\rho(x/\eps)$.
Let~$u_\eps:= u*\rho_\eps$ and~$f_\eps:=f*\rho_\eps$.
Then
$$ (-\Delta)^s u_\e \ugu f_\e \qquad{\mbox{ in }}\;B_{99/100},$$
as long as $\e$ is small enough.
\end{proposition}

\begin{proof} We know that
$$ (-\Delta)^s(\chi_R u) = f+\eta_R+P_R$$
in $B_1$, in the viscosity sense, with $\eta_R\to0$ in $B_1$ as $R\to+\infty$
and $P_R$ is a polynomial of degree at most $k-1$. As a matter of fact,
by choosing the ``optimal representative'' in Lemma \ref{09sc5tgdwefguv},
we can also suppose that
\begin{equation}\label{89uojcd2eruy}
{\mbox{$\eta_R\to0$ uniformly in $B_1$.}}
\end{equation}
Let also 
$$ v_{R,\e}(x):= (\chi_R u)*\rho_\eps(x).$$
Hence
(see e.g. formula~(3.2)
in~\cite{MR3161511}) in $B_{99/100}$ we have that
$$ (-\Delta)^s v_{R,\e}=f_\e+\eta_R*\rho_\eps+P_R*\rho_\eps.$$
Hence, by Corollary \ref{8IJAJS2wedfv1234},
$$ (-\Delta)^s v_{R,\e} \;{\stackrel{0}{=}}\;f_\e+\eta_R*\rho_\eps+P_R*\rho_\eps.$$
So, by \eqref{9idwfjdv827rutgjsakhdaaaaa2},
we have that
\begin{equation}\label{miw8wchhis:2}
(-\Delta)^s v_{R,\e} \;{\stackrel{k}{=}}\;f_\e+\eta_R*\rho_\eps+P_R*\rho_\eps
.\end{equation}
Now we check that
\begin{equation}\label{miw8wchhis}
{\mbox{$ P_R*\rho_\eps$ is a polynomial of degree at most $k-1$.}}
\end{equation}
For this, we can reduce to the case of monomials, and compute,
for any $\alpha\in\N^n$ with $|\alpha|\le k-1$, that
$$ x^\alpha *\rho_\eps =
\int_{\R^n} (x-y)^\alpha \rho_\eps(y)\,dy
= \sum_{\beta\le \alpha}\left( {\alpha}\atop{\beta}\right) 
x^\beta \int_{\R^n} 
(-y)^{\alpha-\beta}\rho_\eps(y)\,dy,$$
which is a polynomial of degree at most $k-1$. This observation proves \eqref{miw8wchhis}.

Then, from \eqref{0dfhvnoooP}, \eqref{miw8wchhis:2}
and \eqref{miw8wchhis}, we conclude that
\begin{equation}\label{miw8wchhis:3}
(-\Delta)^s v_{R,\e} \;{\stackrel{k}{=}}\;f_\e+\eta_R*\rho_\eps.\end{equation}
Our objective is now to send $R\to+\infty$ and use the stability result
in Theorem \ref{STAB}. 
To this aim, we define
$$ v^\star_\e(x):= \int_{\R^n} |u(y)|\, \rho_\eps(x-y)\,dy.$$
We observe that
\begin{equation}\label{miw8wchhis:4}
\int_{B_{99/100}^c} \frac{v_\e^\star(y)}{|y|^{n+2s+k}}\,dy<+\infty,
\end{equation}
see \eqref{HJ01:X03}. Moreover, we have that
\begin{equation}\label{miw8wchhis:5}
v_{R,\e}(x)\le
\int_{\R^n}\big| (\chi_R u)(y)\big|\,\rho_\eps(x-y)\,dy\le
v_\e^\star(x).\end{equation}
In addition
$$ \big| (\chi_R u)(y)\big|\,\rho_\eps(x-y) \le \e^{-n}\,|u(y)| \,\chi_{B_\e(x)}(y)\le
(|x|+1)^{n+2s+k}\frac{|u(y)|}{\e^n\,|y|^{n+2s+k}}.$$
This and \eqref{09sdugfnklsfe245} allow us to use the Dominated Convergence
Theorem and take the limit as~$R\to+\infty$ (for a fixed~$\e>0$).
In this way, we see that, for any fixed $x\in\R^n$,
\begin{equation*}
\lim_{R\to+\infty} v_{R,\e}(x) =
\int_{\R^n} \lim_{R\to+\infty} (\chi_R u)(y) \rho_\eps(x-y)\,dy=
u_\e(x).
\end{equation*}
This, \eqref{miw8wchhis:4} and \eqref{miw8wchhis:5} allow us to use
again the Dominated Convergence
Theorem to take the limit as~$R\to+\infty$ and obtain that
\begin{equation}\label{9e2oy3riufhj}
\lim_{R\to+\infty} \int_{B_{99/100}^c} \frac{|v_{R,\e}(x) -u_\e(x)|}{|x|^{n+2s+k}}\,dx
=0.
\end{equation}
Also, $v_{R,\e}\to u_\e$ and, in view of \eqref{89uojcd2eruy},
$f_\e+\eta_R*\rho_\eps\to f_\e$
locally uniformly in $B_{99/100}$ as $R\to+\infty$.
{F}rom this, \eqref{miw8wchhis:5}
and \eqref{9e2oy3riufhj}, we can exploit
Theorem \ref{STAB} and deduce from \eqref{miw8wchhis:3} that
$$ (-\Delta)^s u_\e \;{\stackrel{k}{=}}\;f_\e$$
in $B_{99/100}$, as desired.
\end{proof}

\begin{appendix}

\section{Appendix A. Summary of the finite differences method}

We recall here the classical method of the finite differences (or incremental quotients).
Given $\omega\in \R^n$, we consider the shift operator acting on functions,
namely $T_\omega f(x):=f(x+\omega)$. Of course, if $\omega=0$, this operator
boils down to the identity operator, which will be denoted by $I$.

For any $h\in(0,1)$ and $\omega\in S^{n-1}$, we set
$$ D_h^\omega:= T_{h\omega}-I.$$
Then, for any $(\omega_1,\dots,\omega_d)\in (S^{n-1})^d$
and any $h\in(0,1)$ we consider the finite difference operator
$$ D_h^{ (\omega_1,\dots,\omega_d) } := D_h^{\omega_1}\,\dots\,D_h^{\omega_d}.$$
Notice that, since the shift operators commute with themselves,
we also have that $ D_h^{ (\omega_1,\dots,\omega_d) } = D_h^{\omega_d}\,\dots\,D_h^{\omega_1}$.

The finite differences of order $d$ approximate the derivatives of order $d$
(after a renormalization of size $h^d$), as pointed out in the following result:

\begin{lemma}\label{DERIVE}
There exists $\xi:\R^n\to [0,1]^d$ such that
$$ h^{-d}\;D_h^{ (\omega_1,\dots,\omega_d) } f(x)=
\sum_{1\le i_1,\dots i_d\le n} \frac{\partial^d f}{\partial x_{i_1}\dots \partial x_{i_d}}
\big(x+h\xi_1(x)\omega_1+\dots+h\xi_d(x)\omega_d\big)\;
\omega_{1 i_1} \dots\omega_{d i_d}.$$
\end{lemma}

\begin{proof} We argue by induction over $d$.
When $d=1$, we use the Mean Value Theorem and we see that
$$ D_h^{\omega_1} f(x) = f(x+h\omega_1)-f(x)=
\nabla f(x+h\xi_1(x)\omega_1)\cdot (h\omega_1),$$
for some $\xi_1:\R^n\to [0,1]$. 

This is the desired claim when $d=1$. Hence, we now suppose that the claim is true
for $d-1$ and we prove it for $d$. For this,
we assume that
\begin{eqnarray*}&& h^{1-d}\;D_h^{ (\omega_1,\dots,\omega_{d-1}) } f(x)\\ &=&
\sum_{1\le i_1,\dots i_{d-1}\le n} \frac{\partial^{d-1} f}{
\partial x_{i_1}\dots \partial x_{i_{d-1}}}
\big(x+h\xi_1(x)\omega_1+\dots+h\xi_{d-1}(x)\omega_{d-1}\big)\;
\omega_{1 i_1} \dots\omega_{d-1 i_{d-1}}\end{eqnarray*}
and we use the Mean Value Theorem to see that
\begin{eqnarray*} &&\frac{\partial^{d-1} f}{
\partial x_{i_1}\dots \partial x_{i_{d-1}}}
\big(x+h\xi_1(x)\omega_1+\dots+h\xi_{d-1}(x)\omega_{d-1}+h\omega_d\big)
\\ &&\qquad-
\frac{\partial^{d-1} f}{
\partial x_{i_1}\dots \partial x_{i_{d-1}}}
\big(x+h\xi_1(x)\omega_1+\dots+h\xi_{d-1}(x)\omega_{d-1}\big)
\\ &=&
\sum_{i_d=1}^n
\frac{\partial^{d} f}{
\partial x_{i_1}\dots \partial x_{i_{d}}}
\big(x+h\xi_1(x)\omega_1+\dots+h\xi_{d}(x)\omega_{d}\big)\,(h\omega_{d i_d}),\end{eqnarray*}
for some $\xi_d:\R^n\to [0,1]$. These observations easily imply the desired claim.
\end{proof}

We also give the following integration by parts formula:

\begin{lemma}\label{ojskdbv63748596}
Let $f\in L^1(\R^n)$ and $g\in L^\infty(\R^n)$. Then
$$ \int_{\R^n} D^{(\omega_1,\dots,\omega_d)}_h f(x)\, g(x)\,dx=
\int_{\R^n} f(x)\,D^{(-\omega_1,\dots,-\omega_d)}_h g(x)\, dx
.$$\end{lemma}

\begin{proof} We argue by induction on $d$. If $d=1$, then
\begin{eqnarray*}
&& \int_{\R^n} D^{\omega_1}_h f(x)\, g(x)\,dx=
\int_{\R^n} f(x+h\omega_1)\, g(x)\,dx
-\int_{\R^n} f(x)\, g(x)\,dx\\
&&\qquad\qquad=
\int_{\R^n} f(x)\, g(x-h\omega_1)\,dx
-\int_{\R^n} f(x)\, g(x)\,dx =
\int_{\R^n} f(x)\,D^{-\omega_1}_h g(x)\,dx,
\end{eqnarray*}
as desired.

For the inductive step, we compute recursively that
\begin{eqnarray*}
&& \int_{\R^n} 
D^{(\omega_1,\dots,\omega_d)}_h f(x)\, g(x)\,dx=
\int_{\R^n}
D^{(\omega_1,\dots,\omega_{d-1})}_h D^{\omega_d}_hf(x)\, g(x)\,dx
\\ &&\qquad\qquad=
\int_{\R^n}
D^{\omega_d}_hf(x)\, D^{(-\omega_1,\dots,-\omega_{d-1})}_h g(x)\,dx
=
\int_{\R^n}
f(x)\, D^{-\omega_d}_h
D^{(-\omega_1,\dots,-\omega_{d-1})}_h g(x)\,dx
\\ &&\qquad\qquad=
\int_{\R^n} f(x)\,D^{(-\omega_1,\dots,-\omega_d)}_h g(x)\, dx,\end{eqnarray*}
which is the desired result.\end{proof}

\section{Appendix B. Proof of Lemmata~\ref{9uids2345hv}
and~\ref{45678959uids2345hv}}

One proof of Lemma~\ref{9uids2345hv}
can be done by exploiting the
finite incremental quotients of order~$d$
(as discussed in Appendix A), to show that~$D^d F$ vanishes.

Another simple, and more geometric, argument is based on the idea
that polynomials are, after all, a finite dimensional space,
and finite dimensional spaces are closed, with respect to any equivalent norm.
The details are the following.

\begin{proof}[Proof of Lemma~\ref{9uids2345hv}]
Up to a translation, we suppose that
\begin{equation}\label{0inu}
0\in U.\end{equation}
Also, without loss of generality, we can suppose that 
\begin{equation}\label{larger}
{\mbox{$m$
in the statement of Lemma \ref{9uids2345hv} is larger than $d$.}}\end{equation}
We define~$N$ to be the number of multi-indices~$\mu\in\N^n$
for which~$|\mu|:=\mu_1+\dots+\mu_n\le d-1$.
In this way, we can endow~$\R^N$ with an ordering
and consider the map~$T$ from~$\R^N$ to the space of polynomials
of degree at most~$d-1$, which is given by
$$ \R^N \ni a=\{ a_\mu\}_{|\mu|\le d-1} \longmapsto
T(a):=\sum_{|\mu|\le d-1} a_\mu x^\mu.$$
We fix distinct points~$q_1,\dots,q_d\in U\subseteq\R^n$.
Then,
on~$\R^N$, we consider the two norms
\begin{eqnarray*}&&
\| a \|_1 := \sum_{i=1}^d |T(a)(q_i)|\\
{\mbox{and }}&&
\| a \|_2 := \| T(a) \|_{C^m(U)}.
\end{eqnarray*}
It is interesting to remark that~$\|\cdot\|_1$
is indeed a norm. For this, suppose that~$\| a \|_1 =0$.
Then, it follows that~$T(a)(q_i)=0$ for any~$i=1,\dots,d$,
hence the polynomial~$T(a)$, which has degree at most~$d-1$,
vanishes on~$d$ different points, and so it has to be zero,
which in turn implies that~$a=0$.

We also write
$$ P^{(j)}=\sum_{|\mu|\le d-1} a_\mu^{(j)} x^\mu,$$
with~$a^{(j)}=\{ a_\mu^{(j)} \}_{|\mu|\le d-1}$.
We remark that
$$ T(a^{(j)}) = P^{(j)}.$$
Therefore,
given~$\eta>0$, if~$j$, $j'\in\N$ are 
sufficiently large (possibly in dependence of~$\eta$), we have that
$$ \| a^{(j)}-a^{(j')} \|_1\le\eta,$$
thanks to~\eqref{6y89ufidhgg177171}, and so~$a^{(j)}$
is a Cauchy sequence in~$\R^N$, with respect to the norm~$\|\cdot\|_1$.

{F}rom the equivalence of the norms in~$\R^N$, it thus follows that~$a^{(j)}$
is a Cauchy sequence in~$\R^N$, with respect to the norm~$\|\cdot\|_2$.
Consequently, given~$\eta>0$, if~$j$, $j'\in\N$ are
sufficiently large,
$$ \eta \ge 
\| a^{(j)}-a^{(j')} \|_2=
\| P^{(j)}- P^{(j')} \|_{C^m(U)}.$$
Therefore, we have that~$P^{(j)}$ is a sequence of functions
that is of Cauchy type in~$C^m(U)$, and so it converges
to some function~$P^\star$ in~$C^m(U)$.

In particular, the sequence~$P^{(j)}$ is bounded in~$C^m(U)$.
{F}rom this and \eqref{larger} we obtain that,
for any~$\mu\in\N^n$ with~$|\mu|\le d-1$,
$$ \sup_{j\in\N} \| P^{(j)}\|_{C^m(U)}\ge
\sup_{j\in\N} \| D^\mu P^{(j)}\|_{L^\infty(U)} \ge
| D^\mu P^{(j)} (0)| = \mu!\,| a^{(j)}_\mu|,$$
thanks to~\eqref{0inu}.
Hence, for any~$\mu\in\N^n$ with~$|\mu|\le d-1$,
up to a subsequence, we have that~$a^{(j)}_\mu\to a^{\star}_\mu$
as $j\to+\infty$, for some~$a^{\star}_\mu\in\R$.
Thus, possibly passing to a subsequence and using~\eqref{6y89ufidhgg177171},
we have that, for any~$x\in U$,
$$ F(x)= \lim_{j\to+\infty} P^{(j)}(x)=\lim_{j\to+\infty}
\sum_{|\mu|\le d-1} a^{(j)}_\mu x^\mu =
\sum_{|\mu|\le d-1} a^{\star}_\mu x^\mu
,$$
that is the desired result.
\end{proof}

\begin{proof}[Proof of Lemma~\ref{45678959uids2345hv}]
We use the setting given by the proof of
Lemma~\ref{9uids2345hv}, and
we define the norm in~$\R^N$ given, for $a=\{a_\mu\}_{|\mu|\le d-1}$, by
$$ \|a\|_3:=
\sup_{{\varphi\in C^2_0(U)}\atop{\|\varphi\|_{C^2(U)}
\le 1}} \int_{U} \sum_{|\mu|\le d-1} a_\mu\,x^\mu\,\varphi(x)\,dx.$$
We see that~$a_\mu^{(j)}$ is a Cauchy sequence with respect to the norm~$\|\cdot\|_3$
and so it converges to some~$a^\star=\{a_\mu^\star\}_{|\mu|\le d-1}\in\R^N$,
with respect to the norm~$\|\cdot\|_3$.

{F}rom the equivalence of the norms in~$\R^N$, we conclude that
$$ 0=\lim_{j\to+\infty}\| a_\mu^{(j)} - a_\mu^{\star}\|_2=
\lim_{j\to+\infty} \| P^{(j)}-T(a^\star)\|_{C^m(U)},$$
which implies the desired result.
\end{proof}

\section{Appendix C. Reabsorbing lower order norms}

The scope of this appendix is to show the following result:

\begin{lemma}\label{Xuax182rrrr3}
Let $k\in\N$, $\gamma\in(0,+\infty)\setminus\N$,
with 
\begin{equation} \label{SFON}
\lfloor\gamma\rfloor\le k-1,\end{equation}
and $f\in C^\gamma(B_1)$. Then
$$ \inf \|f-P\|_{L^\infty(B_1)}+[f-P]_{C^\gamma(B_1)}
\le C\,[f]_{C^\gamma(B_1;k)},$$
where the $\inf$ is taken over all the polynomials~$P$
of degree $k-1$, and $C>0$ depends on $n$, $\gamma$ and $k$.
\end{lemma}

\begin{proof} We write $\gamma=m+\theta$, with $m:=\lfloor\gamma\rfloor\in\N$ and 
$\theta\in(0,1)$. 
We set \begin{eqnarray*}&& T_f(x):=\sum_{|\alpha|\le m} 
\frac{\partial^\alpha f(0)}{\alpha!}x^\alpha ,\\ && G_{f,\alpha}(x):= m\int_0^1 
(1-t)^{m-1} \Big( \partial^\alpha f(tx)-\partial^\alpha f(0)\Big)\,dt \\ {\mbox{and }}&& 
R_f(x):= \sum_{|\alpha|=m} G_{f,\alpha}(x)\,\frac{x^\alpha}{\alpha!}.\end{eqnarray*} 
We observe that $T_f$, $G_{f,\alpha}$ and $R_f$ are linear with respect to $f$
and, in particular,
$$ R_{f+g}=R_f+R_g,$$
for any functions $f$ and $g$.

Notice also that, if $|\alpha|=m$, $$ [ G_{f,\alpha}]_{C^\theta(B_1)} \le C [D^m 
f]_{C^\theta(B_1)}\le C [f]_{C^\gamma(B_1)}.$$ Moreover, a Taylor expansion of $f$ gives that 
$$ f= T_f + R_f.$$ Fix also a generic polynomial $P$ of degree at most $k-1$ of the form 
\begin{eqnarray*} && P(x)=P_1(x)+P_2(x),\\ {\mbox{with }}&&P_1(x):=\sum_{|\alpha|\le m} 
p_\alpha x^\alpha\\ {\mbox{and }}&& P_2(x):=\sum_{|\alpha|\in [m+1,k-1]} p_\alpha 
x^\alpha.\end{eqnarray*} Then, using the $\inf$ notation in the statement of Lemma 
\ref{Xuax182rrrr3}, \begin{equation} \begin{split} \label{9osckv67rugjvcbxnxnxnow} 
\inf_{P} \|f-P\|_{L^\infty(B_1)}+[f-P]_{C^\gamma(B_1)}\,&= \inf_{P_1,P_2} 
\|f-P_1-P_2\|_{L^\infty(B_1)}+[f-P_1-P_2]_{C^\gamma(B_1)}\\ &\le \inf_{P_2} 
\|f-T_f-P_2\|_{L^\infty(B_1)}+[f-T_f-P_2]_{C^\gamma(B_1)} \\&= \inf_{P_2} 
\|R_f-P_2\|_{L^\infty(B_1)}+[R_f-P_2]_{C^\gamma(B_1)} \\ &= \inf_{P_2} 
\|R_{f-P_2}\|_{L^\infty(B_1)}+[R_{f-P_2}]_{C^\gamma(B_1)} .\end{split} \end{equation} 
Now, we observe that, for any function $g$ and any $x\in B_1$, $$ [g]_{C^\gamma(B_1)} \ge 
\frac{|D^m g(x)-D^m g(0)|}{|x|^\theta} \ge |D^m g(x)-D^m g(0)|.$$ 
Since $D^m R_h(0)=0$ for any function $h$,
we can apply the latter estimate with 
$g:=R_{f-P_2}$ and find that $$ [R_{f-P_2}]_{C^\gamma(B_1)} \ge |D^m R_{f-P_2}(x)|,$$ and 
thus, taking supremum over $x\in B_1$, 
\begin{equation} \label{9aixchtwfghwgbrf}
[R_{f-P_2}]_{C^\gamma(B_1)} \ge \| D^m R_{f-P_2}\|_{L^\infty(B_1)}.\end{equation} 
Now we observe that, for any function $g$ with $g(0)=0$
one has that
$$ \|g\|_{L^\infty(B_1)}=\sup_{x\in B_1}
|g(x)|=\sup_{x\in B_1} |g(x)-g(0)|\le \|\nabla g\|_{L^\infty(B_1)}.$$
Since, for any function $h$, it holds that $D^j R_h(0)=0$
for any $j\in \{0,\dots,m-1\}$, we can apply this estimate
repeatedly and find that
$$ \|R_h\|_{L^\infty(B_1)} \le C\,\|\nabla R_h\|_{L^\infty(B_1)}
\le \dots\le C\, \|D^m R_h\|_{L^\infty(B_1)},$$
up to renaming $C>0$.

{F}rom this and \eqref{9aixchtwfghwgbrf}, we obtain
$$ \| R_{f-P_2}\|_{L^\infty(B_1)}
\le C\, [ R_{f-P_2}]_{C^\gamma(B_1)}.$$
So, we insert 
this information into \eqref{9osckv67rugjvcbxnxnxnow} and we obtain 
\begin{equation}\label{7dhdhhd812is22} \begin{split}&\inf_{P} 
\|f-P\|_{L^\infty(B_1)}+[f-P]_{C^\gamma(B_1)} \le 
2\,\inf_{P_2}[R_{f-P_2}]_{C^\gamma(B_1)}\\ &\qquad\quad= 2\,\inf_{P_2}[R_{f}-R_{P_2}]_{C^\gamma(B_1)} 
=2\,\inf_{P_2}[R_{f}-{P_2}]_{C^\gamma(B_1)} .\end{split}
\end{equation} We also remark that, since 
$\gamma>m$, it holds that
$$ [h-\bar P-P_2]_{C^\gamma(B_1)} =[h-P_2]_{C^\gamma(B_1)},$$ for any 
function $h$ and any polynomial $\bar P$ of degree at most $m$, hence 
\eqref{7dhdhhd812is22} gives that $$ \inf_{P} 
\|f-P\|_{L^\infty(B_1)}+[f-P]_{C^\gamma(B_1)} \le 2\,\inf_{P_2}[R_{f}-\bar 
P-{P_2}]_{C^\gamma(B_1)}.$$ We choose now $\bar P:= \bar Q - T_f$, where $\bar Q$ is a 
generic polynomial of degree at most $m$. In this way, we obtain $$ \inf_{P} 
\|f-P\|_{L^\infty(B_1)}+[f-P]_{C^\gamma(B_1)} \le 2\,\inf_{P_2}[R_{f}+T_f-\bar 
Q-{P_2}]_{C^\gamma(B_1)}= 2\,\inf_{P_2}[f-\bar Q-{P_2}]_{C^\gamma(B_1)}.$$ Since $\bar 
Q+P_2$ is now the generic polynomial of degree at most $k-1$ (notice indeed that $m\le 
k-1$, in view of \eqref{SFON}), we obtain $$ \inf_{P} 
\|f-P\|_{L^\infty(B_1)}+[f-P]_{C^\gamma(B_1)} \le 2\,\inf_{P}[f-P]_{C^\gamma(B_1)},$$ as 
desired (recall \eqref{CGAMMA}). \end{proof}

\end{appendix}

\end{document}